\documentclass[11pt]{article}
\usepackage[style=numeric]{biblatex}
\usepackage{amsmath,amssymb,amsfonts}
\usepackage{mathrsfs}
\usepackage{amsthm}
\usepackage{xcolor}
\usepackage{svg}
\usepackage[makeroom]{cancel}
\usepackage[ruled,vlined]{algorithm2e}
\usepackage{ulem}
\usepackage{url}
\usepackage[T1]{fontenc}
\usepackage{caption}
\usepackage{a4wide}
\DeclareMathOperator*{\argmin}{\arg\!min}
\DeclareMathAlphabet{\mathpzc}{OT1}{pzc}{m}{it}

\newcommand{\review}[1]{#1}

\usepackage{tikz}
 
\definecolor{darkgreen}{rgb}{0.2, 0.4, 0.23}

\begin{document}

\newcommand{\operator}{\mathcal{F}}
\newcommand{\mub}{\boldsymbol{\mu}}
\newcommand{\nub}{\boldsymbol{\nu}}
\newcommand{\xb}{\mathbf{x}}
\newcommand{\vb}{\mathbf{v}}
\newcommand{\wb}{\mathbf{w}}
\newcommand{\ub}{\mathbf{u}}
\newcommand{\ufomp}{\ub_{\mub,\nub}}
\newcommand{\uromp}{\ufomp^{\textnormal{ROM}}}
\newcommand{\fomdim}{N_{h}}
\newcommand{\pgeo}{\Theta}
\newcommand{\pphys}{\Theta'}
\newcommand{\ngeo}{p}
\newcommand{\nphys}{p'}
\newcommand{\dod}{\mathbf{V}}
\newcommand{\VV}{\mathbb{V}}
\newcommand{\mass}{\mathbb{G}}
\newcommand{\solmanifold}{\mathscr{S}}
\newcommand{\grassmann}{\mathscr{G}}
\newcommand{\fomspace}{\mathbb{R}^{N_{h}}}
\newcommand{\dodout}{\mathbb{R}^{N_{h}\times n}}
\newcommand{\spann}{\text{span}}
\newcommand{\distance}{\mathpzc{d}}
\newcommand{\ntrain}{N_{\text{train}}}
\newcommand{\loss}{\mathscr{L}}
\newcommand{\ubc}{\boldsymbol{u}}
\newtheorem{theorem}{Theorem}
\newtheorem{lemma}{Lemma}
\newtheorem{remark}{Remark}
\newcommand{\adapt}{\textnormal{Adpt}}
\newcommand{\ambient}{\mathbb{A}}
\newcommand{\nambient}{N_{A}}
\newcommand{\hidefig}[1]{}
\renewcommand{\arraystretch}{1.25}
\newcommand{\tobedone}{\textcolor{red}{To be done.}}

\newtheorem{corollary}{Corollary}



\title{Deep orthogonal decomposition: a  continuously adaptive data-driven approach to model order reduction} 

\author{Nicola Rares Franco$^{a}$\footnote{Corresponding author.\\\indent\;\;\;E-mail: nicolarares.franco@polimi.it, Postal address: via Edoardo Bonardi, 9, 20133 Milano MI, Italy}, Andrea Manzoni$^{a}$, Paolo Zunino$^{a}$, Jan S. Hesthaven$^{b}$}

\date{}
\maketitle

\vspace{-0.5cm}
\begin{center}
\noindent $^a$ \textit{MOX, Department of Mathematics, Politecnico di Milano, P.zza Leonardo da Vinci, 32, Milan, 20094, Italy}\vspace{-0.25cm}
\end{center}
\begin{center}
$^b$ \textit{Computational Mathematics and Simulation Science, École Polytechnique Fédérale de Lausanne, Station 8, Lausanne, 1015, Switzerland}
\end{center}

%\vspace{-0.5cm}
\begin{abstract}
We develop a novel deep learning technique, termed Deep Orthogonal Decomposition (DOD), for dimensionality reduction and reduced order modeling of parameter dependent partial differential equations. The approach consists in the construction of a deep neural network model that approximates the solution manifold through a continuously adaptive local basis. In contrast to global methods, such as Principal Orthogonal Decomposition (POD), the %new
adaptivity allows the DOD to overcome the Kolmogorov barrier, making the approach applicable to a wide spectrum of parametric problems. Furthermore, due to its hybrid linear-nonlinear nature, the DOD can accommodate both intrusive and nonintrusive techniques, providing highly interpretable latent representations and tighter control on error propagation. For this reason, the proposed approach stands out as a valuable alternative to other nonlinear techniques, such as deep autoencoders. The methodology is discussed both theoretically and practically, evaluating its performances on problems featuring nonlinear PDEs, singularities, and parametrized geometries.
\end{abstract}


\vspace{-0.25cm}


\section{Introduction}
\label{sec:intro}

Thanks to the massive advancements in computer science and numerical analysis over the past century, intricate partial differential equations (PDEs) can nowadays be solved with incredible accuracy. This breakthrough allows simulation of complex physical phenomena, with applications ranging from aerospace %engineering
\cite{farhat2001multidisciplinary} to cardiovascular engineering \cite{bucelli2023mathematical}. 

However, producing accurate simulations can require a large amount of computational resources, sometimes involving hours to days of processing time. While this may be acceptable for certain applications, it is definitely unsuitable whenever dealing with many-query scenarios, where numerous solutions need to be computed sequentially and in real-time. For instance, in the context of parameter-dependent PDEs, applications involving inverse problems \cite{cao2023residual, benfenati2024modular, lahivaara2019estimation,
mucke2023markov,
zabaras2011solving}, optimal control \cite{bader2016certified, kleikamp2023greedy, ravindran2000reduced, strazzullo2018model} or uncertainty quantification \cite{cicci2023uncertainty, vitullo2024deep, zhu2018bayesian}, can easily become prohibitive due to prolonged processing times and unbearable computational demands. As a remedy, Reduced Order Modeling techniques \cite{hesthaven2016certified, quarteroni2016reduced} are now gaining the attention of many scientists and researchers.

In simple terms, Reduced Order Modeling aims at creating effective model surrogates, known as reduced order models (ROMs), capable of mimicking the precision of traditional methods in price of a lower computational cost. Usually, these ROMs achieve their proficiency by learning from high quality simulations obtained by means of classical numerical solvers, commonly referred to as full order models (FOMs), extracting essential information for replicating the behavior of complex system. This learning process allows ROMs to capture the essential features of the original models, making them capable of providing accurate results at a fraction of the computational expense. However, this increased efficiency is typically achieved at the cost of an acceptable reduction in model accuracy.

As of today, the literature has been constantly evolving, to the point that domain practitioners can now rely on a very broad spectrum of %reduced order modeling 
reduction techniques. In general, different approaches can provide different guarantees, and choosing the appropriate ROM is often a matter of the underlying application. For instance, data-driven ROMs based on Machine Learning and Deep Learning algorithms, such as \cite{hesthaven2018non, franco2023deep, fresca2021pod, pichi2024graph}, can be extremely efficient, but they can be unsuited if the application demands a strong agreement with the physical constraints characterizing the system under study -- despite several efforts towards this direction; see, e.g., \cite{boon2023deep, chen2021physics, goswami2023physics}. Conversely, other ROMs directly work with the governing equation, reducing their complexity via suitable projection techniques, which makes them intrinsically coherent with the underlying physics, see, e.g. \cite{benner2015survey, hesthaven2016certified, quarteroni2016reduced}. In turn, these ROMs can become extremely inefficient when dealing with complex systems described in terms of highly nonlinear PDEs, or when facing problems with space-localized effects or moving fronts \cite{ohlberger2015reduced}.

However, regardless of this distinction, most ROMs are grounded on a common foundation, which is to rely on \textit{dimensionality reduction} techniques. The latter consist of suitable algorithms capable of compressing and reconstructing high-fidelity simulations by mapping them onto a small feature space, also known as the \textit{latent space}. To this end, researchers can rely on a large pletora of different approaches, from linear techniques based on orthogonal projections, such as Principal Orthogonal Decomposition (POD) \cite{quarteroni2016reduced}, to fully nonlinear strategies, employing, e.g., wavelet transforms \cite{devore1993wavelet}, deep autoencoders \cite{franco2023deep} and more. Once again, the choice is typically problem-specific: in some cases, linear projection techniques can provide a high-level of accuracy at a high compression rate, as in the case of diffusion processes; in other situations, instead, nonlinear techniques are more favorable as, although returning much complex representations, they can effectively capture nontrivial features such as sharp edges, moving fronts, and singularities.

Mathematically speaking, this distinction is perfectly represented by the concept of Kolmogorov $n$-width. Given a target set $\mathscr{S}$, which we can think of as the set that collects all PDE solutions associated with a given parametric problem, and a reduced dimension $n$, the Kolmogorov $n$-width $d_{n}(S)$ measures the extent to which the set $S$ can be approximated using linear subspaces of dimension $n$. %: see Eq. \eqref{eq:kolmogorov} in the next Section for a more rigorous definition.
If, for a certain problem, $d_{n}(\mathscr{S})$ decays rapidly in $n$, then linear techniques can be a favorable option. With simple representations and many theoretical guarantees, projection-based methods are most likely the best alternative when dealing with such situations. However, if $d_{n}(\mathscr{S})$ decays slowly, linear methods become unsuitable, as, in order to ensure a reasonable accuracy, they require a large latent space, which, in turn, undermines the usefulness of actual dimensionality reduction. Then, in this case, nonlinear techniques can be a potential solution.

In this work, we would like to focus on a specific scenario that, compared to the previous ones, is arguably somewhere in between (see Figure \ref{fig:kolmogorovs} for positioning of this work within a synthesis of the state-of-the-art). To illustrate the idea, consider an ideal example of fluid flow around an obstacle. For simplicity, let us assume a low Reynolds regime, so that the whole dynamics is well represented by the Stokes equations. The problem may depend on several parameters: some of them, which we collect in the parameter vector $\mub$, determine the position/orientation of the obstacle, while others, denoted as $\nub$, influence physical quantities, such as, for example, the fluid viscosity and/or the inflow condition.
As testified by multiple works in the ROM literature, see, e.g., \cite{apacoglu2011cfd, lorenzi2016pod}, assuming a fixed geometric configuration $\mub_{0}$, linear methods can be extremely effective in capturing the variability of the fluid flow with respect to $\nub$. However, if we allow the obstacle location to change, that is, if we let $\mub$ vary within a suitable parameter space, the situation changes dynamically: linear methods begin to struggle, and the Kolmogorov $n$-width starts to deteriorate, reflecting the typical behavior of parametric PDEs with space-interacting parameters \cite{ohlberger2015reduced, romor2023non}. %. \textcolor{red}{forse qualche citazione in piu ci starebbe}

In general, this is a prototypical example of a situation in which: (i) $d_{n}(\mathscr{S})$ decays slowly, but (ii) the solution manifold admits a decomposition into suitable submanifolds, $\mathscr{S_{\mub}}$, with fast decaying Kolmogorov $n$-widths. In particular, these submanifold can be obtained by fixing the values of the %\guillemotleft bad\guillemotright
"bad" parameters, $\mub$, while leaving the others, $\nub$, free to change.
Our purpose for this work is to present a novel approach to model order reduction that can exploit the intrinsic regularity of such problems by leveraging on a \textit{continuously adaptive} linear subspace. Specifically, the idea is to construct a deep neural network architecture %$\dod$ 
that, to each parameter vector $\mub$, associates a corresponding linear subspace $\mathcal{V}_{\mub}$ approximating the submanifold $\mathscr{S}_{\mub}\subset\mathscr{S}$. Equivalently, our proposal can be seen as a continuous generalization of localized POD algorithms. Given the use of neural networks and the intimate connection with classical POD, we term our approach Deep Orthogonal Decomposition (DOD).
\\\\
The paper is organized as follows. First, in Section~\ref{sec:setup}, we start by setting some notation, introducing the problem of interest, the underlying assumptions, and their consequences. Then, in Section~\ref{sec:dod}, we introduce the DOD algorithm, discussing the whole idea from the sole perspective of dimensionality reduction; the use of DOD for reduced order modeling, instead, is deferred to Section~\ref{sec:dodrom}. In both cases, we evaluate the proposed approaches through numerical experiments, comparing their performances with the state-of-the-art. These results are reported in Section~\ref{sec:exp1} and~\ref{sec:exp2}, respectively. Lastly, we devote Section~\ref{sec:conclusions} to a concluding discussion, where we underscore both the strengths and limitations of the proposed approach, offering additional information on possible future developments.

\begin{figure}
    \centering
    \includegraphics[width = 1\textwidth]{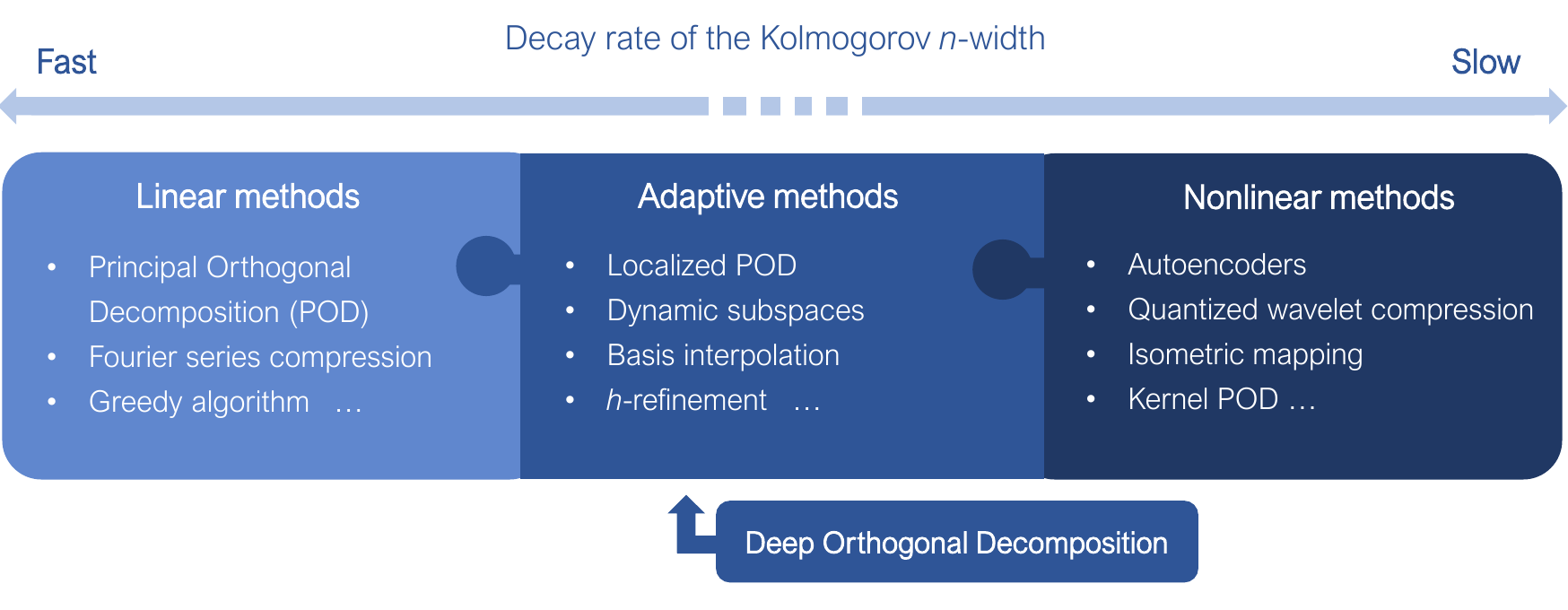}
    \caption{A non-comprehensive list of dimensionality reduction techniques in reduced order modeling. While linear techniques can be very effective for problems with a fast-decaying Kolmogorov $n$-width, nonlinear techniques are generally better suited in the opposite scenario. Problems "in between" can present mixed properties, such as a general slow decay compensated by a much faster one in the submanifolds.}
    \label{fig:kolmogorovs}
\end{figure}

\begin{remark}
    It has come to our attention that the term "Deep Orthogonal Decomposition" is not entirely new: in fact, the same terminology can be found in \cite{tait2020deep}, an unpublished work by Daniel J. Tait (2020). In principle, the two approaches are entirely different, as they pursue fundamentally different goals: ours focuses on \emph{parametrized stationary PDEs}, leveraging on a suitable decoupling of the parameter space for constructing a continuously adaptive basis; \cite{tait2020deep}, instead, deals with \emph{(unparametrized) time-dependent PDEs}, leveraging on neural networks for the construction of a memory-aware time-adaptive local basis. On top of this, the two approaches differ in terms of neural network architectures, training routines, and online computations. However, given that both works share the common idea of exploiting deep learning for the construction of an adaptive local basis, we insist on using the same terminology, possibly paving the way towards new avenues in model order reduction. 
\end{remark}

\section{Problem setup}
\label{sec:setup}
We start by fixing some notation.
Let $(V_{h},\|\cdot\|_{V_{h}})$ be a finite-dimensional Hilbert state space, $V_{h}\cong\mathbb{R}^{N_{h}}$, arising, for instance, from a suitable Finite Element discretization of a given stationary PDE, so that $V_{h}\subset L^{2}(\Omega)$ for some spatial domain $\Omega\subset\mathbb{R}^{d}$. Let $\{\varphi_{i}\}_{i=1}^{N_{h}}$ be a basis for $V_{h}$. Given $u\in V_{h}$, we write $$\mathbf{u}:=[\mathrm{u}^{(1)},\dots,\mathrm{u}^{(N_{h})}]^{\top}$$ for the corresponding vector of degrees of freedom (dof), that is, the set of basis coefficients that produce the representation of $u$ in terms of $\varphi_{1},\dots,\varphi_{N_{h}}$, in the sense that
\begin{equation}
    \label{eq:fom-basis}
    u(\mathbf{x}) = \sum_{i=1}^{N_{h}}\mathrm{u}^{(i)}\varphi_{i}(\mathbf{x})\quad\quad\forall\mathbf{x}\in\Omega.
\end{equation}
Let $\langle\cdot,\cdot\rangle_{V_{h}}$ be the inner product associated to $\|\cdot\|_{V_{h}}$. We define the Gram matrix $\mass\in\mathbb{R}^{N_{h}\times N_{h}}$ as the symmetric positive definite matrix given by
\begin{equation}
    \label{eq:gram}
    \mass:=\left[\begin{array}{ccc}
       \langle\varphi_{1},\;\varphi_{1}\rangle_{V_{h}} & ...  & \langle\varphi_{1},\; \varphi_{N_{h}}\rangle_{V_{h}} \\
       ...  & ... & ... \\
       \langle\varphi_{N_{h}}, \varphi_{1}\rangle_{V_{h}} & ...  & \langle\varphi_{N_{h}}, \varphi_{N_{h}}\rangle_{V_{h}}
    \end{array}\right].
\end{equation}
When $\|\cdot\|_{V_{h}}=\|\cdot\|_{L^{2}}$ is the $L^{2}$ norm, the latter is commonly referred to as the \textit{mass} matrix. The Gram matrix allows us to equip $\mathbb{R}^{N_{h}}$ with the following norm
$$\|\mathbf{u}\|:=\sqrt{\mathbf{u}^{\top}\mass\mathbf{u}},$$
which is nothing but the discrete equivalent of $\|\cdot\|_{V_{h}}$. In fact, it is easy to see that $\|\mathbf{u}\|=\|u\|_{V_{h}}$ whenever $\mathbf{u}$ is the dof representation of $u$. Note that this norm can differ substantially from the Euclidean norm $$|\mathbf{u}|:=\sqrt{\mathbf{u}^{\top}\mathbf{u}}.$$ 
In particular, the two coincide if and only if $\varphi_{1},\dots,\varphi_{N_{h}}$ are orthonormal with respect to $\langle\cdot,\cdot\rangle_{V_{h}}$, in which case $\mass=\mathbb{I}$ is the identity matrix.
\\\\
With this setup, let us now introduce the parametric problem.
Let $\mub\in\mathbb{R}^{\ngeo}$ and $\nub\in\mathbb{R}^{\nphys}$ be two vectors of parameters. Ideally, we collect in $\mub$ all parameters that have a geometric or space-varying nature; all the remaining parameters, which typically concern the physical properties of the model at hand, are collected in $\nub$. We allow both $\mub$ and $\nub$ to vary within a suitable parameter space, herein assumed to be compact: we shall write $\mub\in\pgeo$ and $\nub\in\pphys.$ We consider a parametrized problem of the %following 
form: given $(\mub,\nub)\in\pgeo\times\pphys$ find $u\in V_{h}$ such that
\begin{equation}
\label{eq:fom}
\mathscr{R}(\mub,\nub,u) = 0,\end{equation}
where $\mathscr{R}: \pgeo\times\pphys\times V_{h}\to \mathbb{R}$ is a given parameter dependent nonlinear operator, inclusive of external quantities such as, e.g., boundary conditions or source terms. We think of \eqref{eq:fom} as a discretized parameter dependent stationary PDE, formulated in a weak or strong form, with $\mathscr{R}$ representing the norm of the PDE residual.

%Having fixed a suitable basis for $V_{h}$, say $\{\varphi_{i}\}_{i=1}^{N_{h}}$, 
By leveraging the dof representation in \eqref{eq:fom-basis}, problem \eqref{eq:fom} naturally defines a map from $\pgeo\times\pphys\to\fomspace$ which maps every parameter combination onto the basis coefficients of the corresponding PDE solution, namely
$$\operator: (\mub,\nub)\mapsto \ufomp:=[\mathrm{u}^{(1)}_{\mub,\nub},\dots, \mathrm{u}^{(N_{h})}_{\mub,\nub}]$$
such that
$u_{\mub,\nub}:=\sum_{i=1}^{N_{h}}\mathrm{u}^{(i)}_{\mub,\nub}\varphi_{i}$
solves \eqref{eq:fom}. Our purpose is to provide an efficient approximation of this parameter-to-solution operator $\operator$, so that we can avoid repeated calls to the PDE solver when a large number of evaluations are required.
\\\\
In general, changes in model parameters may produce different effects on the PDE solution. Here, we address a specific scenario, which is easily explained via the concept of \textit{Kolmogorov $n$-width}. %Let $\|\cdot\|$ denote the norm over $\fomspace$ induced by $\|\cdot\|_{V_{h}}$. 
Given a set $S\subset\fomspace$, we define its Kolmogorov $n$-width $d_{n}(S)$ as the error achieved by its "best approximation" in terms of linear subspaces of dimension $n$, that is,
\begin{equation}
    \label{eq:kolmogorov}
    d_{n}(S):=\inf_{\VV\in\mathbb{R}^{N_{h}\times n}}\;\sup_{\ub\in S}\;\|\ub-\VV\VV^{\top}\mass\ub\|.
\end{equation}
Note in fact that if $\VV$ is $\mass$ orthonormal, which means that $\VV^{\top}\mass\VV$ is the identity matrix $n\times n$, then $\VV\VV^{\top}\mass\ub$ is the projection of $\ub$ onto $\spann(\VV)$. %with respect to $\|\cdot\|.$
%
%where $\mass\in\mathbb{R}^{N_{h}\times N_{h}}$ is the \textit{Gram matrix} encoding the inner product over $\fomspace$ as induced by $(V_{h},\|\cdot\|_{V_{h}})$. 
In this work, we focus our attention on those cases in which:
\begin{itemize}
    \item [A1.] the solution manifold $\solmanifold:=\{\ufomp\;|\;\mub\in\pgeo,\nub\in\pphys\}\subset V_{h}$ exhibits a slow-decay of the Kolmogorov $n$-width, e.g., $$d_{n}(\solmanifold)\le Cn^{-\alpha}$$ for some $C>0$ and $\alpha\in(0,1)$;
    \item [A2.] the geometrical/space-varying parameters are the main cause to (A.1), in the sense that the submanifolds $\solmanifold_{\mub}:=\{\ufomp\;|\;\nub\in\pphys\}\subset\solmanifold$, corresponding to $\mub$-slices of $\solmanifold$, have a uniformly fast decay, e.g.
    $$\sup_{\mub\in\pgeo}\;d_{n}(\solmanifold_{\mub}) \le C'n^{-\beta}$$%e^{-\beta n}$$
    for some $C'>0$ and $\beta\ge1$. %independent of $\mub$.%$C',\beta>0$ independent of $\mub$.
\end{itemize}

This scenario can be extremely common in parametrized problems featuring space-parameter interaction. To illustrate this, we report below a simple, yet remarkable, example.

\subsection{An instructive example}

Let us consider a stationary 2D fluid flow, modeled through the Navier-Stokes equations, that occurs in the spatial domain depicted in Figure \ref{fig:navier-stokes-domain}.a. The fluid flows from left to right, passing around an almond-shaped obstacle whose diameter is roughly a fifth of the channel width. We consider a parametrized scenario depending on five scalar parameters:
\begin{itemize}
    \item $\nub=[\alpha,\beta]$, with $0\le\alpha,\beta\le10$, which parametrize the inflow condition at boundary $\Gamma_{in}$, imposing a Dirichlet condition of the form
    $$\boldsymbol{u}(0,y)=y(1-y)\left(\alpha e^{-100(y-0.25)^{2}}+\beta e^{-100(y-0.75)^{2}}\right)^{1/2}\quad\quad\forall y\in[0,1].$$
    Larger values of $\alpha$ correspond to a stronger flow at the bottom, while larger values of $\beta$ increase the fluid velocity at the top;
    \item $\mub=[\theta,x_{0},y_{0}]$, which parametrize the rotation, $0\le\theta\le2\pi$, and the location of the obstacle, $0.25\le x_{0},y_{0}\le 0.75$.
\end{itemize}

For simplicity, we consider a FOM solver based on a Finite Element discretization combined with a fictitious domain approach, so that, for all admissible parameter configurations $\mub,\nub$, the corresponding solutions can be sought within the same state space $V_{h}\cong\mathbb{R}^{N_{h}}$: for further details, we refer the interested reader to Section \ref{subsec:navier-stokes}.

Intuitively, it is clear that the two vectors of parameters, $\mub$ and $\nub$, play fundamentally different roles. The former are more geometrical in nature and can significantly affect the behavior of the solution in multiple ways. The variability introduced by the remaining parameters is, instead, much simpler to describe, as they can only distribute the flow intensity either at the top or at the bottom. For example, if we fix a geometric configuration $\mub$, and let $\nub$ vary, we can speculate that the corresponding fluid flows will be given, roughly speaking, by the superposition of two main modes: one describing the flow at the top and one describing the flow at the bottom. In this sense, the submanifold $\mathscr{S}_{\mub}$, which consists of all PDE solutions for fixed $\mub$ and variable $\nub$, should be well approximated by a small linear subspace. In contrast, if we allow $\mub$ to change, this nice behavior is likely to disappear.
\\\\
In particular, this is also what we observe in practice. 
Figure \ref{fig:navier-stokes-domain}.b shows the decay of the projection error for increasingly large linear subspaces, comparing the behavior of the whole solution manifold $\mathscr{S}$ with that of the submanifolds $\mathscr{S}_{\mub}$, computed for different values of $\mub$. Here, in order to estimate these trends, we relied on Proper Orthogonal Decomposition: while we refrain from delving into technical details in this context, readers keen on a more rigorous explanation can refer to Section \ref{sec:exp1}.

Even from a qualitative point of view, it is clear that Figure \ref{fig:navier-stokes-domain}.b confirms our intuition. Furthermore, a more thorough examination suggests that Assumptions A1 and A2 are true: in fact, a least-square regression in logarithmic scale yields a decay rate of 0.688 for $\mathscr{S}$, and a much steeper slope for the submanifolds $\mathscr{S}_{\mub}$ (exponent ranging from 1.958 to 2.454).
\\\\
In general, although its simplicity, this case study highlights how certain problems strongly call for an adaptive basis approach and how common this situation can be.

\begin{figure}[tb]
    \centering
    \includegraphics[width=0.45\textwidth]{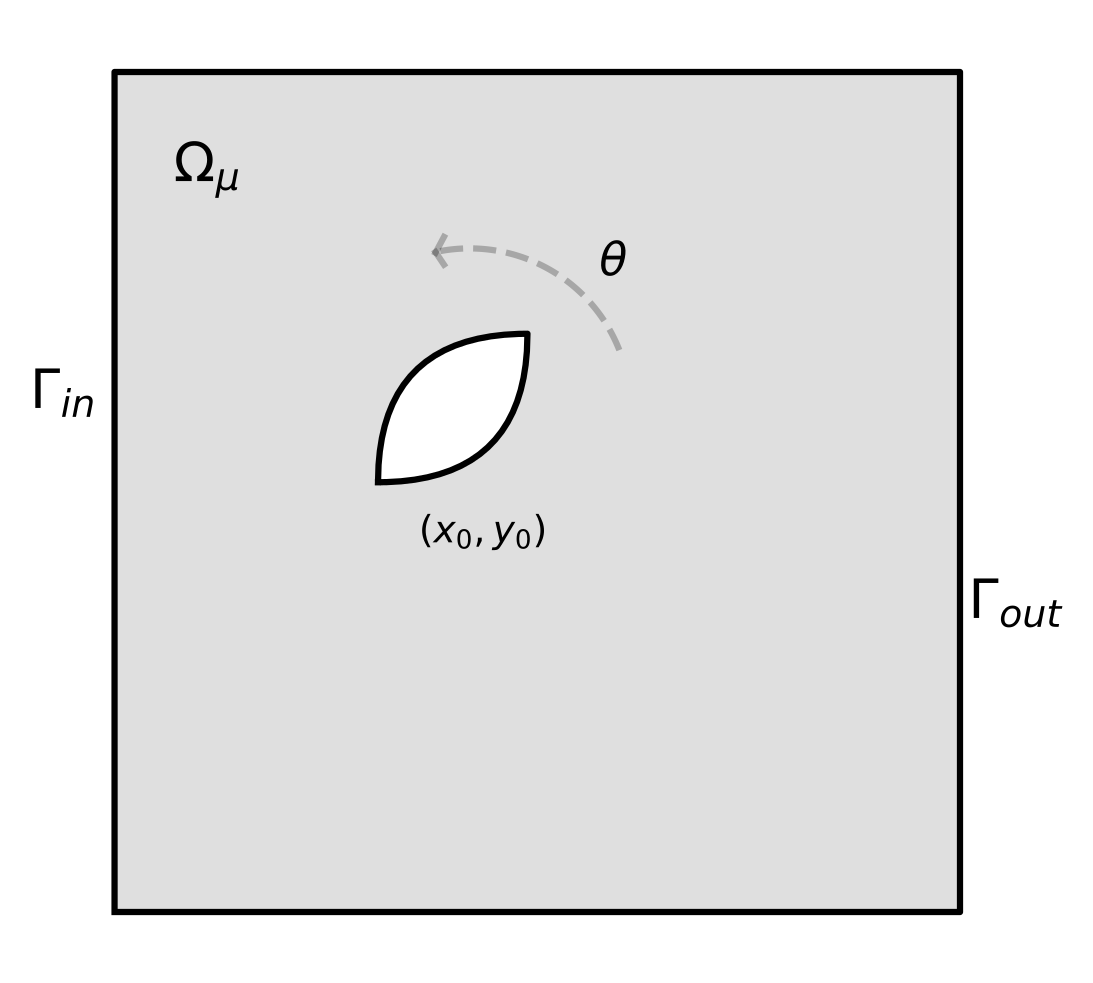}
   \includegraphics[width=0.45\textwidth]{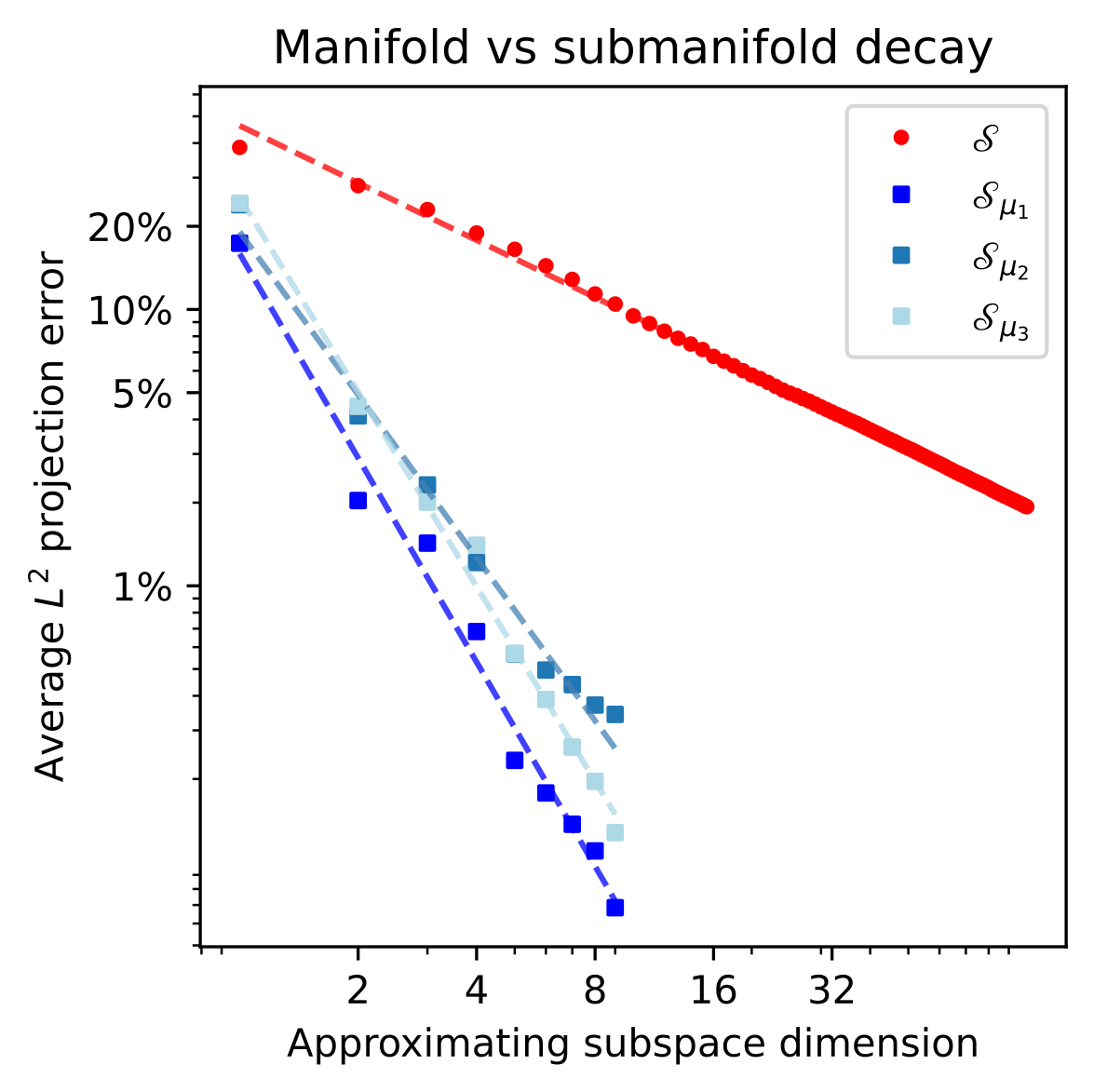}
    \caption{Spatial domain (left) and projection error analysis (right) for the Navier-Stokes case study, Sections \ref{subsec:example} and \ref{subsec:navier-stokes}. Left: the domain $\Omega_{\mub}$ is obtained by removing an almond-shaped object from the unit square $(0,1)^{2}$. The parameters $\mub=[\theta,x_{0},y_{0}]$ determine the rotation and the position of the obstacle. Right: decay of the projection error for increasingly larger subspaces, highlighting the differences between the solution manifold $\mathscr{S}$ and its $\mub$-slices $\mathscr{S}_{\mub}$. Here, $\mub_{1}=[0,0.5,0.5]$, $\mub_{2}=[\pi/4,0.4,0.6]$ and $\mub_{3}=[\pi/2,0.7,0.3].$}
    \label{fig:navier-stokes-domain}
\end{figure}

\subsection{Existing techniques in the ROM literature}

To further motivate the adoption of a novel approach, it may be worth discussing how traditional ROMs could tackle problems such as the one described in Section \ref{subsec:example}, emphasizing their inherent shortcomings and limitations. 
In general, as we already mentioned, ROMs based on linear reduction techniques would face major difficulties because of the harsh behavior in the Kolmogorov $n$-width. For instance, non-intrusive approaches falling into this category, such as POD-NN \cite{hesthaven2018non}, POD-GPR \cite{guo2018reduced}, or POD-DeepONet \cite{lu2022comprehensive}, would be forced to use a large number of basis functions, resulting in highly complex ROMs that may require many training data; see, e.g., \cite{hesthaven2018non}. In addition to this, intrusive ROMs such as POD-Galerkin \cite{hesthaven2016certified, quarteroni2016reduced} would face additional challenges posed by the presence of nonlinear terms in the governing equations (as with the Navier-Stokes equations).

Clearly, ROMs based on nonlinear dimensionality reduction algorithms, such as deep autoencoders, may mitigate this fact. However, these approaches would completely ignore the peculiar structure of the problem, treating the solution manifold as a whole. Due to this, fully nonlinear methods based on the encoding-decoding paradigm would require at least $\dim(\pgeo)+\dim(\pphys)=p+p'$ latent variables to effectively represent PDE solutions, cf. \cite[Theorem 3]{franco2023deep}. Furthermore, coupling these methods with intrusive techniques can be highly nontrivial, which often brings domain practitioners to completely forget about the underlying physics of the system, pushing them towards purely data-driven alternatives.
Instead, given the particular structure characterizing such problems, a far more natural approach would be to rely on adaptive basis techniques. 

In this regard, the ROM literature provides several alternatives, which, to the best of our knowledge, can be summarized as falling into one of the following categories: dictionary-based approaches \cite{bonito2021nonlinear, geelen2022localized, pagani2018numerical}, basis interpolation methods \cite{amsallem2008interpolation}, and time-adaptive techniques \review{\cite{peherstorfer2020model, hesthaven2022rank, hesthaven2023adaptive}}.  Here, given our focus on stationary PDEs, we shall limit our discussion to the first two classes. %of approaches. 
\review{Particularly, our attention will be limited to \textit{offline} adaptive approaches, that is, ROMs relying on a family of local basis that are learned during training; \textit{online} adaptive schemes, in contrast, implement suitable strategies that allow ROMs to flexibly adapt as new parametric configurations are encountered, or as time flows. Thus, one of their main purposes is to anticipate unseen behaviors. We refer the interested reader to \cite{peherstorfer2015online, singh2023lookahead}.}
\\\\
For dictionary-based ROMs, the idea is to approximate the solution manifold through a finite number $c>1$ of (distinct) low-dimensional subspaces, typically represented by a list of orthonormal matrices $\mathbb{V}_{1},\dots,\mathbb{V}_{c}\in\mathbb{R}^{N_{h}\times n}$. These subspaces can be identified in several ways, for example, by relying on clustering algorithms, as in \cite{pagani2018numerical}, or domain decomposition techniques, as in \cite{bonito2021nonlinear}. In both cases, the output of the ROM will depend discontinuously on $\mub$, which, in some situations, can be rather undesirable. Furthermore, large parameter spaces can pose significant challenges, as clustering-based ROMs need to address the classification problem 
$$\mub\;\mapsto\;\argmin_{i\in\{1,\dots, c\}}\;\sup_{\nub\in\pphys}\|\ub_{\mub,\nub}-\mathbb{V}_{i}\mathbb{V}_{i}^{\top}\ub_{\mub,\nub}\|,$$
while domain decomposition methods can struggle in partitioning $\pgeo.$

On the contrary, basis interpolation techniques, such as the angle interpolation method and its generalizations \cite{amsallem2008interpolation}, define a continuously adaptive basis $\mub\mapsto\mathbf{V}_{\mub}$ by integrating suitable interpolation routines. That is, starting from a finite set of subspaces, much like in the previous case, they construct labeled pairs of the form $\{(\mub_{i},\mathbb{V}_{i})\}_{i=1}^{c}$. The data is then interpolated in a suitable way so that, given $\mub\in\pgeo$, the corresponding local basis $\mathbf{V}_{\mub}$ can be easily computed. To ensure that $\mathbf{V}_{\mub}$ is an orthonormal basis, such interpolation is carried out across the so-called Grassmann manifold (cf. Section \ref{subsec:adaptivity}) rather than in the Euclidean space $\mathbb{R}^{N_{h}\times n}$. For example, in the case of 1-dimensional parameter spaces, if $\mu_{1}<\mu<\mu_{2}$, the interpolation routine combines the local basis $\mathbb{V}_{1}$ and $\mathbb{V}_{2}$ using the formula

$$\mathbf{V}_{\mu}:=\mathbb{V}_{1}\mathbb{X}\cos\left[\left(\frac{\mu-\mu_1}{\mu_2-\mu_1}\right)\tan^{-1}(\mathbb{S})\right]+\mathbb{Y}\sin\left[\left(\frac{\mu-\mu_1}{\mu_2-\mu_1}\right)\tan^{-1}(\mathbb{S})\right]$$
where $\mathbb{Y}\mathbb{S}\mathbb{X}^{\top}=(\mathbb{I}-\mathbb{V}_{1}\mathbb{V}_{1}^{\top})\mathbb{V}_{2}(\mathbb{V}_{1}^{\top}\mathbb{V}_{2})^{-1}$ are computed via SVD, and trigonometric functions act component-wise: see, e.g., \cite{amsallem2008interpolation}. However, although fairly general, these approaches can have a hard time whenever $p$ becomes mildly large, facing the major limitations posed by interpolation methods in high-dimensions, such as the well-known curse of dimensionality.

Here, in order to overcome these limitations, we propose an alternative approach based on deep neural networks, \review{powerful approximators that have gained traction across many scientific domains, including, most recently, adaptive ROMs (see, e.g., \cite{berman2024colora}).}

\label{subsec:example}

\section{Deep Orthogonal Decomposition for dimensionality reduction}
\label{sec:dod}

We devote this Section to the presentation of the DOD algorithm, first discussing its applicability in the broader context of dimensionality reduction. A subsequent discussion on the use of DOD for reduced order modeling will be provided in Section \ref{sec:dodrom}.
Following the notation in Section \ref{sec:setup}, let $$\operator:\pgeo\times\pphys\ni(\mub,\nub)\mapsto \ufomp\in\mathbb{R}^{N_{h}}$$
be a parameter-to-solution operator, where the model parameters have been subdivided into two groups: those responsible for the slow-decay in the Kolmogorov $n$-width of $\solmanifold=\operator(\pgeo\times\pphys)$, collected in the vector $\mub$, and those whose effect is rapidly captured by linear combinations, stored in $\nub$. The idea is to construct a deep neural network model, called the DOD, %Deep Orthogonal Decomposition (DOD), 
which is capable of parametrizing a suitable modal basis $\dod=\dod_{\mub}$ that changes adaptively with $\mub$ in a highly efficient manner. The reason for this is that we would like to take advantage of the nice behavior of the variables $\nub$ as much as possible, while simultaneously isolating the difficulty of handling $\mub$. %the dependency on $\mub$.
In practice, having fixed a latent dimension $n$, we seek for a suitable DNN architecture
\begin{equation}
    \dod:\pgeo\longrightarrow\mathbb{R}^{N_{h}\times n}
\end{equation}
$$\;\;\mub\longrightarrow\dod_{\mub}$$
such that, for any given $\mub\in\pgeo$, the matrix $\dod_{\mub}$ acts as a good local basis for the submanifold $\solmanifold_{\mub}=\{\ufomp\}_{\nub\in\pphys}=\operator(\mub,\pphys)$. From a quantitative point of view, this boils down to requiring %that
$$\ufomp\approx\dod_{\mub}\dod_{\mub}^{\top}\mass\ufomp$$
in $\|\cdot\|$-norm. Then, if the projection error is sufficiently small, the $n$-dimensional vector of DOD coefficients $$\mathbf{c}_{\mub,\nub}:=\dod_{\mub}^{\top}\mass\ufomp$$
can be used as a proxy for the overall highfidelity solution $\ufomp\in\mathbb{R}^{\fomdim}.$ In fact, the latter can be easily recovered via the lifting $\mathbf{c}_{\mub,\nub}\mapsto\dod_{\mub}\mathbf{c}_{\mub,\nub}.$

As exemplified by the theoretical result below, this procedure is mathematically sound, and further motivated whenever Assumptions A1 and A2 are satisfied. We defer the proof to the Appendix so as to avoid an excessive deviation from the main topic. In what follows, we use $\mathbb{E}$ to denote the expectation operator. More precisely, given two probability distributions $\mathbb{P}$ and $\mathbb{Q}$, defined over $\pgeo$ and $\pphys$, respectively, for any measurable map $f:\pgeo\times\pphys\to[0,+\infty]$ we let
$$\mathbb{E}_{\mub,\nub}[f(\mub,\nub)]:=\int_{\pgeo}\int_{\pphys} f(\mub,\nub)\mathbb{P}(d\mub)\mathbb{Q}(d\nub).$$
Similarly, given $g:\pgeo\to[0,+\infty]$, we set $\mathbb{E}_{\mub}[g(\mub)]:=\int_{\pgeo}g(\mub)\mathbb{P}(d\mub).$

\begin{theorem}
\label{theorem:dod}
Let $\pgeo\subset\mathbb{R}^{p}$ and $\pphys\subset\mathbb{R}^{p'}$ be two compact sets, equipped, respectively, with two probability distributions, $\mathbb{P}$ and $\mathbb{Q}$.
Let $$\operator: \pgeo\times\pphys\ni(\mub,\nub)\to \ufomp\in\mathbb{R}^{\fomdim}$$ be continuous. For each $\mub\in\Theta$, let $\solmanifold_{\mub}:=\{\ufomp\}_{\nub\in\pphys}\subset\operator(\pgeo\times\pphys)$ be the $\mub$-submanifold in the image of $\operator$. Let $\mass$ be the Gram matrix associated with a given inner product in $\fomspace$, and let $\|\cdot\|$ be the corresponding norm.
Then, for every $\varepsilon>0$ there exists a ReLU matrix-valued deep neural network $\dod:\mathbb{R}^{p}\to\mathbb{R}^{\fomdim\times n}$ such that
$$\mathbb{E}_{\mub,\nub}\|\ufomp-\dod_{\mub}\dod_{\mub}^{\top}\mass\ufomp\|<\varepsilon+\mathbb{E}_{\mub}\left[d_{n}(\solmanifold_{\mub})\right],$$
where $\dod_{\mub}:=\dod(\mub)$.\end{theorem}
\begin{proof}
    We refer the interested reader to the Appendix.
\end{proof}

Clearly, Theorem \ref{theorem:dod} is only an existence result. In particular, it does not provide an answer to three main questions: i) how to construct such a network, ii) how to train it, and iii) whether the overall approach is computationally feasible. We shall start by answering the first two questions, while we leave the third one to the numerical experiments, Section \ref{sec:exp1}.

\begin{remark}
In this work, we adopt a fully algebraic perspective, as that can be more natural in the context of model order reduction. However, to better understand the overall idea, it may be useful to discuss the implications of the DOD approach at the continuous level. To this end, we note that the parameter-dependent PDE solution can be described as a map $u=u(\xb,\mub,\nub)$, depending explicitly on the space variable and on the model parameters. At the continuous level, a classical POD decomposition corresponds to a separation of variable approach of the form
\begin{equation}
    \label{eq:poddecoupling}
    u(\xb,\mub,\nub)\approx\sum_{i=1}^{n}v_{i}(\xb)\phi_i(\mub,\nub),
\end{equation}
where $n$ is the reduced dimension, $v_{i}$ corresponds to the $i$th mode (represented by the $i$th column in the POD matrix), while $\phi_{i}(\mub,\nub)$ is the corresponding parameter dependent coefficient. With this formalism, the DOD approach can be regarded as
\begin{equation}
    \label{eq:doddecoupling}
    u(\xb,\mub,\nub)\approx\sum_{i=1}^{n}v_{i}(\xb,\mub)\phi_i(\mub,\nub),
\end{equation}
effectively presenting a $\mub$-adaptive basis. From an intuitive point of view, Eq. \eqref{eq:doddecoupling} is emphasizing the fact that the "space-interacting" parameters, $\mub$, should not be decoupled from the space variable, $\xb$, when approximating $u$.
\end{remark}

\subsection{Architecture design}
\label{sec:design}
Having to deal with remarkably large dimensions, from $p$ to $\fomdim\times n$, the construction of a DOD architecture requires some discussion. In principle, the high dimension at output could be tackled by relying on suitable layer types, specifically designed for handling high-dimensional data, such as, e.g., convolutional models (CNNs) \cite{franco2023approximation}, graph neural networks (GNNs) \cite{battaglia2018relational}, or mesh-informed neural networks (MINNs) \cite{franco2023mesh}. However, all these approaches have limited scalability: as of today, using these architectures to address problems with $N_{h}\sim 10^{4}-10^{6}$ degrees of freedom requires a significant amount of computational resources, often beyond practical feasibility.

For this reason, and to be as general as possible, we propose a simpler approach, based on the introduction of a suitable \textit{ambient space}, approximating the original state space. Simply put, we start by introducing an ambient matrix $\ambient\in\mathbb{R}^{\fomdim\times \nambient}$, where $\nambient$ is smaller than $N_{h}$ but still fairly large, e.g., $N_{A}\sim 10^{2}-10^{3}$, such that $\ambient$ is $\mass$-orthonormal and
\begin{equation*}
    \ufomp\approx\ambient\ambient^{\top}\mass\ufomp,
\end{equation*}
with a given tolerance. In practice, $\ambient$ can be constructed by computing a preliminary POD over the training snapshots: see, e.g., Algorithm \ref{algo:ambient}. However, we remark that this is only a preliminary reduction whose purpose is to make the FOM data more manageable; by no means, we assume $N_{A}$ to be small: the actual reduction in dimensionality will be carried out by the DOD. In a way, this intermediate step corresponds to rewriting FOM solutions in a more convenient, problem-specific, way: in fact, while the FOM basis is defined \textit{a priori}, the ambient space is constructed \textit{a posteriori} by leveraging on the training data.
We also note that, in spirit, this is the same trick adopted by other %deep learning 
techniques, such as POD-DL-ROM \cite{fresca2021pod} %, POD-NN \cite{hesthaven2018non}
and POD-DeepONet \cite{lu2022comprehensive}.
\begin{algorithm}[t]
\SetAlgoLined
\normalem
\SetKwInOut{Input}{Input}
\SetKwInOut{Output}{Output}
\vspace{0.25em}
\Input{List of training simulations $[\ub_{1},\dots,\ub_{\ntrain}]$, Gram matrix $\mass$, ambient dimension $N_{A}$.
\vspace{0.5em}
}

\Output{Ambient matrix $\ambient$.
\vspace{0.75em}
}

 \textit{// Preprocessing}\vspace{0.4em}
 
 $\mathbb{U}\leftarrow$ stack $[\ub_{1},\dots,\ub_{\ntrain}]$\vspace{0.4em}

 $\mathbb{M}\leftarrow \mathbb{U}^{\top}\mass\mathbb{U}$\vspace{0.5cm}

\textit{// Eigenvalues and eigenvectors, with $\lambda_{i}\ge\lambda_{i+1}$}\vspace{0.4em}

 $\lambda_{1},\dots,\lambda_{\ntrain}$ and $
\boldsymbol{\xi}_{1},\dots,\boldsymbol{\xi}_{\ntrain}\leftarrow$ eig$(\mathbb{M})$\vspace{0.5cm}

\textit{// Modes truncation}\vspace{0.4em}

$\boldsymbol{\Lambda}\leftarrow$ diag$(\lambda_{1},\dots,\lambda_{N_{A}})$ \vspace{0.4em}

$\Xi\leftarrow$ stack $[\boldsymbol{\xi}_{1},\dots,\boldsymbol{\xi}_{N_{A}}]$\vspace{0.4em}

$\ambient\leftarrow \mathbb{U}\Xi\boldsymbol{\Lambda}^{-1/2}.$\vspace{0.5em}

\Return{$\ambient$}\vspace{0.25em}

 \caption{\label{algo:ambient}Construction of the ambient space via generalized POD (arbitrary inner product).}
 
\end{algorithm}
The main advantage of this maneuver is that the DOD network can now be constructed as
$$\dod_{\mub}:=\ambient\tilde{\dod}_{\mub},$$
with $\tilde{\dod}:\mathbb{R}^{\ngeo}\to\mathbb{R}^{\nambient\times n}$ being the learnable component of the architecture. In particular, since $N_{A}$ can be orders of magnitude smaller than $N_{h}$, adopting this approach can substantially reduce the number of trainable parameters in the DOD, thus simplifying its design and optimization.
%, thus effectively replacing the original FOM dimension with that of the ambient space. 
%
We call $\tilde{\dod}$ the \textit{inner module} of the DOD. 

Note that $\dod_{\mub}$ is $\mass$-orthonormal if and only if $\tilde{\dod}_{\mub}$ is orthonormal in the Euclidean sense, in fact,
$$\dod_{\mub}^{\top}\mass\dod_{\mub}=\tilde{\dod}_{\mub}^{\top}\ambient^{\top}\mass\ambient\tilde{\dod}_{\mub}=\tilde{\dod}_{\mub}^{\top}\tilde{\dod}_{\mub}.$$
To construct the matrix-valued network $\tilde{\dod}$, we use a composite architecture comprised of 
%As for other deep learning approaches, such as DL-ROM \cite{franco2023deep}, the complexity of this task can be decoupled in two parts. The first one, concerns the development of a suitable computational unit that process the input parameters, here given by $\mub\in\mathbb{R}^{p}.$ The second one, instead, involves the design of a proper terminal block that can tackle the high dimension at output. 
%Taking these factors into account, we propose to design DODs according to the following scheme (see also Figure \ref{fig:dod} for a visual representation). The idea is to split the architecture in three parts:
\begin{itemize}
    \item[i)] a \textit{seed} module, $s:\mathbb{R}^{p}\to\mathbb{R}^{l}$, whose purpose is to pre-process the input parameters by mapping them onto a suitable feature space;
    \item[ii)] a collection of \textit{root} modules, $R_{1},\dots,R_{n}:\mathbb{R}^{l}\to\mathbb{R}^{\nambient}$ operating in parallel, whose purpose is to compute the several columns of the (inner) DOD projector $\tilde{\dod}$;
    %comprised of $n$ stacked architectures, $R_{1},\dots,R_{n}$, each going from the feature space, $\mathbb{R}^{l}$, onto the high-fidelity space, $\mathbb{R}^{\fomdim}$, .
    \item [iii)] %a \textit{Gram-Schmidt layer},
    an \textit{ORTH unit}, that is, a nonlearnable block ensuring orthormality of the final output. The latter accepts a matrix $\mathbb{W}\in\mathbb{R}^{\nambient\times n}$ and returns a corresponding orthonormal matrix $\tilde{\mathbb{W}}:=\textnormal{ORTH}(\mathbb{W})\in\mathbb{R}^{\nambient\times n}$ such that $\spann(\mathbb{W})=\spann(\tilde{\mathbb{W}}).$ In practice, this can be achieved in many equivalent ways, e.g. via reduced QR decomposition \cite{golub2013matrix} or via Gram-Schmidt orthogonalization \cite{quarteroni2006numerical}.
    %$$\tilde{\mathbb{A}}\mapsto\mathbb{A}:=\text{Gram-Schmidt}(\tilde{\mathbb{A}},\;\text{inner} = \mass).$$
    %By relying on the classical Gram-Schmidt procedure, the latter accepts any matrix $\tilde{\mathbb{A}}\in\mathbb{R}^{N_{h}\times n}$ and produces a corresponding $\mathbb{A}\in\mathbb{R}^{\fomdim\times n}$ such that $\spann(\mathbb{A})=\spann(\tilde{\mathbb{A}})$ and $\mathbb{A}^{\top}\mass\mathbb{A}=\mathbb{I}$, with $\mathbb{I}\in\mathbb{R}^{n\times n}$ the identity matrix. 
\end{itemize}
Both the seed and the root components are implemented via classical deep feed forward neural networks. The overall workflow can be summarized as in Figure \ref{fig:dod} or, in formula, as
$$\dod_{\mub}=\ambient\text{ORTH}\left(\left[R_{1}(s_{\mub}),\dots,R_{n}(s_{\mub})\right]\right),$$
where $s_{\mub}:=s(\mub)$.
%Given an input parameter $\mub\in\pgeo\subset\mathbb{R}^{\ngeo}$, the corresponding DOD matrix at output is computed as
%$$\dod_{\mub}:=\text{Gram-Schmidt}(\tilde{\dod}_{\mub},\;\text{inner} = \mass)\in\mathbb{R}^{\fomdim\times n},$$
%where
%$$\tilde{\dod}_{\mub}:=\left[R_{1}(s_{\mub}),\dots,R_{n}(s_{\mub})\right]\in\mathbb{R}^{\fomdim\times n},$$
%and $s_{\mub}:=s(\mub).$ The idea is that $s$ acts as a "shared" reparametrization of the input, while the roots $R_{1},\dots,R_{n}$ operate independently to retrieve the (unorthonormalized) DOD basis. 
%In general, assuming $\ngeo+l\ll\fomdim$, the seed module, $s$, can be constructed using classical dense architectures. Conversely, the roots $R_{1},\dots, R_{n}$ should be designed in terms of suitable layer types that can handle high dimensional data, such as, convolutional models (CNNs) or graph neural networks (GNNs). Here, we shall rely on mesh-informed neural networks (MINNs) \cite{franco2023mesh}, a new class of sparse architectures that was specifically designed to tackle mesh-based functional data, that is, to address those situations in which the high dimensions arise from the introduction of a Finite Element discretization. For further details about these architectures and their practical implementation, we refer to \cite{franco2023mesh}.

\begin{remark}
    The ambient space is only a practical expedient that we have introduced in order to tackle arbitrarily large FOMs. However, if $N_{h}$ is reasonably small, this step can be omitted and one may work directly at the FOM level. In practice, if $\tilde{\varphi}_{1},\dots,\tilde{\varphi}_{N_{h}}$ is an orthonormal basis derived from $\varphi_{1},\dots,\varphi_{N_{h}}$, that is, from the original FOM basis, this would be equivalent to setting $\ambient:=[\tilde{\boldsymbol{\varphi}}_{1},\dots,\tilde{\boldsymbol{\varphi}}_{N_{h}}]$,    
    so that $\ambient^{\top}\mass\ambient=\mathbb{I}$ and $\textnormal{span}(\ambient)=\mathbb{R}^{N_{h}}$.
\end{remark}

\begin{figure}[t!]
    \centering
    \includegraphics[width=\textwidth]{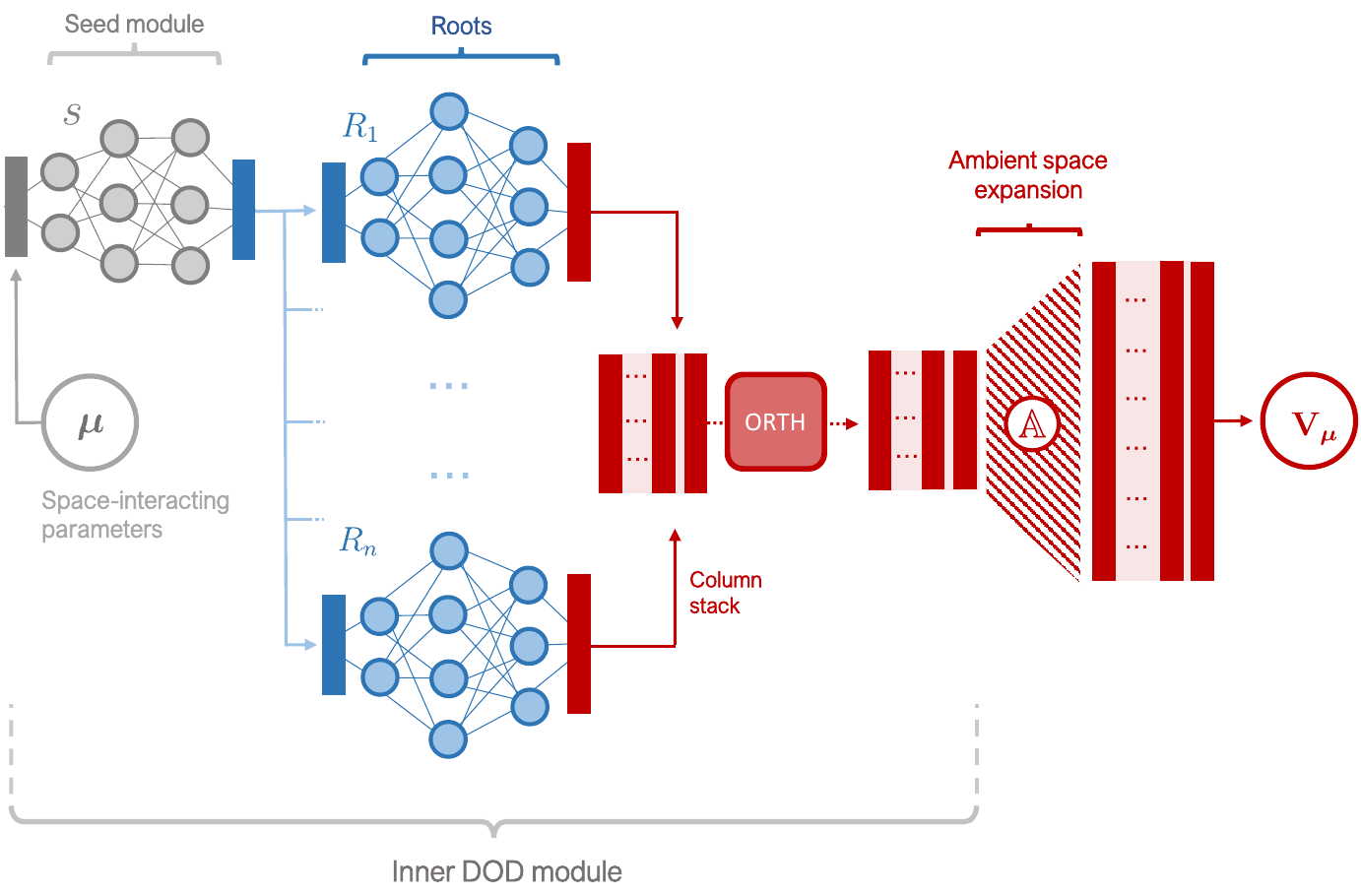}
    \caption{Sketch of a DOD architecture.}
    \label{fig:dod}
\end{figure}

\subsection{Model training}
\label{subsec:training}
In order to learn the DOD basis, we propose a supervised training strategy based on a variational principle, where the DOD architecture $\dod$ is trained by minimizing the reconstruction error over the training data, as depicted in Algorithm \ref{algo:train1}. In other words, we learn the DOD by minimizing the loss function below,
\begin{equation}
    \label{eq:train1}
    \loss(\dod):=\frac{1}{\ntrain}\sum_{i=1}^{\ntrain}\|\ub_{\mub_{i},\nub_{i}}-\dod_{\mub_{i}}\dod_{\mub_{i}}^{\top}\mass\ub_{\mub_{i},\nub_{i}}\|^{2},
\end{equation}
where $\{\mub_{i},\nub_{i},\ub_{\mub_{i},\nub_{i}}\}_{i=1}^{\ntrain}\subset\pgeo\times\pphys\times \mathbb{R}^{N_{h}}$ are highfidelity samples generated -randomly- by repeated calls to the FOM solver.
This approach is fairly intuitive, as it defines the DOD projector following the same minimization principle of POD. %Furthermore, it can be easily implemented as it does not require any particular sampling of the parameter space. %However, as we shall see in Section \ref{sec:exp1}, this approach can entail a few drawbacks if the training snapshots come in very different scales. For this reason, we also provide an alternative training strategy, which we detail in the next subsection.

In practice, since optimizing \eqref{eq:train1} can be computationally demanding, we can take advantage of the existence of the ambient space, $\mathbb{A}$, in order to ease computational effort. In fact, minimizing \eqref{eq:train1} is equivalent to minimizing
\begin{multline}
    \label{eq:train2}
    \loss_{\ambient}(\tilde{\dod}):=\frac{1}{\ntrain}\sum_{i=1}^{\ntrain}|\ambient^{\top}\mass\ub_{\mub_{i},\nub_{i}}-\tilde{\dod}_{\mub_{i}}\tilde{\dod}_{\mub_{i}}^{\top}\ambient^{\top}\mass\ub_{\mub_{i},\nub_{i}}|^{2}=\\=
    \frac{1}{\ntrain}\sum_{i=1}^{\ntrain}|\tilde{\ub}_{\mub_{i},\nub_{i}}-\tilde{\dod}_{\mub_{i}}\tilde{\dod}_{\mub_{i}}^{\top}\tilde{\ub}_{\mub_{i},\nub_{i}}|^{2},
\end{multline}
where we recall that $|\cdot|$ denotes the Euclidean norm. Indeed, the two loss functions only differ by a constant: see Lemma \ref{lemma:equiv} in the following.

\begin{lemma}
    \label{lemma:equiv}
    Let $\ambient\in\mathbb{R}^{\fomdim\times\nambient}$ be $\mass$-orthonormal. Fix a reduced dimension $ n\le \nambient$ and let $\tilde{\dod}:\mathbb{R}^{\ngeo}\to\mathbb{R}^{\nambient\times n}$ be any matrix-valued map. Let $\dod_{\mub}:=\ambient\tilde{\dod}_{\mub}$. Then
    \begin{equation}
    \label{eq:equiv}
    \loss(\dod_{\mub})=c_{\ambient}+\loss_{\ambient}(\tilde{\dod}),\end{equation}
    where $c_{\ambient}>0$ is a constant depending on $\ambient$ and on the training data.
\end{lemma}
\begin{proof}
    Fix any $\mub_{i},\nub_{i},\ub_{\mub_{i}, \nub_{i}}$ in the training set. Since the ambient residual $$\ub_{\mub_{i},\nub_{i}}-\ambient\ambient^{\top}\mass\ub_{\mub_{i},\nub_{i}}$$
    is $\mass$-orthonormal to $\spann(\ambient)$, it follows that
    \begin{multline*}\|\ub_{\mub_{i},\nub_{i}}-\dod_{\mub_{i}}\dod_{\mub_{i}}^{\top}\mass\ub_{\mub_{i},\nub_{i}}\|^{2}=\|\ub_{\mub_{i},\nub_{i}}-\ambient\tilde{\dod}_{\mub_{i}}\tilde{\dod}_{\mub_{i}}^{\top}\ambient^{\top}\mass\ub_{\mub_{i},\nub_{i}}\|^{2}=\\=
    \|\ub_{\mub_{i},\nub_{i}}-\ambient\ambient^{\top}\mass\ub_{\mub_{i},\nub_{i}}\|^{2}+
    \|\ambient\ambient^{\top}\mass\ub_{\mub_{i},\nub_{i}}-\ambient\tilde{\dod}_{\mub_{i}}\tilde{\dod}_{\mub_{i}}^{\top}\ambient^{\top}\mass\ub_{\mub_{i},\nub_{i}}\|^{2}=\\=
    \|\ub_{\mub_{i},\nub_{i}}-\ambient\ambient^{\top}\mass\ub_{\mub_{i},\nub_{i}}\|^{2}+
    |\ambient^{\top}\mass\ub_{\mub_{i},\nub_{i}}-\tilde{\dod}_{\mub_{i}}\tilde{\dod}_{\mub_{i}}^{\top}\ambient^{\top}\mass\ub_{\mub_{i},\nub_{i}}|^{2}.
    \end{multline*}
    Then, averaging over the training set yields \eqref{eq:equiv}.    
\end{proof}

Hence, to summarize, the implementation and training of a DOD network can be carried out as follows. First, following the guidelines presented in Section \ref{sec:design}, we design the model architecture. According to Figure \ref{fig:dod}, this corresponds to fixing an ambient space $\ambient$, a latent dimension $n$, and a set of neural network architectures defining, respectively, the seed and the roots modules. From an abstract point of view, this is equivalent to identifying a suitable hypothesis class
$$\mathscr{D}\subset\{\tilde{\dod}:\pgeo\to\mathbb{R}^{N_{A}\times n}\}.$$
for the inner DOD module. Then, we rely on classical optimization algorithms, such as L-BFGS or Adam, to solve the following minimization problem 
$$\min_{\tilde{\dod}\in\mathscr{D}}\;\loss_{\ambient}(\tilde{\dod}),$$
which, ultimately, corresponds to training the (inner) DOD network.

\begin{algorithm}[t]
\SetAlgoLined
\normalem
\SetKwInOut{Input}{Input}
\SetKwInOut{Output}{Output}
\vspace{0.25em}
\Input{FOM solver $\texttt{FOM}=\texttt{FOM}(\mub,\nub)$, parameter spaces $\pgeo$ and $\pphys,$ inner module architecture class $\mathscr{D}$ of reduced dimension $n$, sample size $\ntrain$, Gram matrix $\mass$, ambient dimension $\nambient$ with $\nambient>n$.
\vspace{0.5em}
}

\Output{Trained DOD model $\dod$.
\vspace{0.75em}
}

 $[\mub_{1},\dots,\mub_{\ntrain}]\leftarrow$ i.i.d. random sample from $\pgeo$\vspace{0.5em}

  $[\nub_{1},\dots,\nub_{\ntrain}]\leftarrow$ i.i.d. random sample from $\pphys$\vspace{0.5em}

 $[\ub_{1},\dots,\ub_{\ntrain}]\leftarrow[\texttt{FOM}(\mub_{i},\nub_{i})\;\text{for}\;i=1:\ntrain]$\vspace{0.5em}\hspace{0.5cm}\textit{// sampling}

 $\ambient\leftarrow \texttt{POD}([\ub_{1},\dots,\ub_{\ntrain}],\;\mass, \;\nambient)$\vspace{1em}\hspace{0.5cm}\textit{// ambient space definition}

$[\tilde{\ub}_{1},\dots,\tilde{\ub}_{\ntrain}]\leftarrow[\ambient^{\top}\mass\ub_{i}\;\text{for}\;i=1:\ntrain]$\vspace{0.75em}\hspace{0.5cm}\textit{// ambient projection}

$\tilde{\dod}_{*}\leftarrow\argmin_{\tilde{\dod}\in\mathscr{D}}\;\frac{1}{\ntrain}\sum_{i=1}^{\ntrain}|\tilde{\ub}_{i}-\tilde{\dod}(\mub_{i})\tilde{\dod}^{\top}(\mub_{i})\tilde{\ub}_{i}|^{2}$\vspace{0.5em}\hspace{0.5cm}\textit{// training}

$\dod\leftarrow \ambient\tilde{\dod}_{*}$\hspace{0.5cm}\textit{// map composition}\vspace{0.5em}

\Return{$\dod$}\vspace{0.25em}

 \caption{\label{algo:train1}Construction and training of the DOD architecture.}
 
\end{algorithm}

\subsection{Quantifying adaptivity}
\label{subsec:adaptivity}
After training, it can be useful to quantify the actual adaptivity of the DOD basis: that is, to which extent the map $\mub\mapsto \dod_{\mub}$ is non-constant over the parameter space $\pgeo.$ As we shall see in the next Section, this postprocessing can substantially increase our understanding of the DOD approach, and help us in designing better models.

In principle, measuring the variability of the DOD basis across $\pgeo$ might seem straightforward: in fact, since the DOD is explicitly given (and in closed form) by a neural network model, we can easily compute quantities such as derivatives, variances, and so on. However, this is not the full story, and things are actually more complicated.
To appreciate this, let us consider a very simple example where $p=1$, $N_{h}=3$ and $\|\cdot\|$ is the Euclidean norm. Consider the DOD model below
\begin{equation}
\label{eq:motivating-grassmann}
    \mu\mapsto \dod_{\mu}:=\left[
\begin{array}{cc}
    \cos\mu & -\sin\mu \\
    \sin\mu & \cos\mu\\
    0 & 0
\end{array}
\right].\end{equation}
At first sight, it may look like the DOD basis is adaptively changing with the input parameter $\mu$. However, this is not really the case. In fact,
$$\spann(\dod_{\mu})=\spann(\dod_{0})$$
for all $\mu\in\mathbb{R}.$ In particular, since the projection error depends only on the underlying subspace (and not on the matrix representation), the maps $\mu\mapsto \mathbf{V}_{\mu}$ and $\mu\mapsto\mathbf{V}_{0}$ are actually equivalent. In this sense, a DOD network acting as \eqref{eq:motivating-grassmann} would not be adaptive at all.
\\\\
These considerations are key, as they bring us to the following observation: the outputs of a DOD network are not matrices, but \textit{subspaces}. As such, the variability of the DOD basis is better understood in terms of the so-called \textit{Grassmann manifold} \review{\cite{amsallem2008interpolation, zimmermann2018geometric, bendokat2024grassmann, wong1967differential}}.
For a given dimension $n$ and a suitable ambient space $\mathcal{A}$, the Grassmann manifold consists of all subspaces $\mathcal{V}\subseteq\mathcal{A}$ of dimension $n$, namely
$$\grassmann_{n}(\mathcal{A}):=\{\mathcal{V}\subseteq\mathcal{A}\;\text{such that}\;\mathcal{V}\;\text{is linear and}\;\dim(\mathcal{V})=n\},$$
defined whenever $1\le n\le \dim(\mathcal{A})$. The Grassmann manifold can be equipped with different metrics; see, e.g. \cite{edelman1998geometry} for a comprehensive list, specifically designed for measuring distances between subspaces. Here, by noting that the variability of $\dod$ ultimately depends on that of its inner module $\tilde{\dod}$, we focus on the case $\mathcal{A}=\mathbb{R}^{\nambient}$ and we endow the Grassmann manifold with the following metric 
$$\distance(\mathcal{V},\mathcal{W}):=\;\max_{\substack{\mathbf{w}\in\mathcal{W}\\|\mathbf{w}|=1}}\;\min_{\mathbf{v}\in\mathcal{V}}\;|\mathbf{w}-\mathbf{v}|,$$
so that $\distance:\grassmann_{n}(\mathbb{R}^{\nambient})\times\grassmann_{n}(\mathbb{R}^{\nambient})\to[0,1]$. Equivalently, if $\tilde{\mathbb{V}}$ and $\tilde{\mathbb{W}}$ are orthonormal matrices representing the two subspaces $\mathcal{V}$ and $\mathcal{W}$, respectively, then $$\distance(\mathcal{V},\mathcal{W})=\sqrt{1-\sigma_{\min}^{2}},$$ where $\sigma_{\min}$ is the smallest singular value of $\tilde{\mathbb{V}}^{\top}\tilde{\mathbb{W}}$, see Algorithm \ref{algo:distance}. In the literature, this metric is also known as \textit{projection 2-norm} \cite{edelman1998geometry}.
\begin{algorithm}[t]
\SetAlgoLined
\normalem
\SetKwInOut{Input}{Input}
\SetKwInOut{Output}{Output}
\vspace{0.25em}
\Input{
Matrices $\mathbb{V},\mathbb{W}\in\mathbb{R}^{\nambient\times n}$.
\vspace{0.25em}
}

\Output{
Metric distance $\distance$ between $\spann(\mathbb{V})$ and $\spann(\mathbb{W})$.
\vspace{0.75em}
}

$\tilde{\mathbb{V}}\;\leftarrow$ ORTH($\mathbb{V}$) $\;\;\;\setminus\setminus$ \textit{Orthonormalization}\vspace{0.5em}

$\tilde{\mathbb{W}}\leftarrow$ ORTH($\mathbb{W}$)\vspace{0.75em}

$[\sigma_{1},\dots,\sigma_{n}]\leftarrow$ singular values of $\tilde{\mathbb{V}}^{\top}\tilde{\mathbb{W}}$ (in decreasing order)\vspace{0.75em}

\Return{$\sqrt{1-\sigma_{n}^{2}}$}\vspace{0.25em}

 \caption{\label{algo:distance}Computation of the distance $\distance$ over the Grassmann manifold $\grassmann_{n}(\mathbb{R}^{\nambient})$.\vspace{0.25em}}
 
\end{algorithm}
\\\\
With this setup, we can think of the inner DOD module as of a map
$$\tilde{\dod}:\pgeo\to\grassmann_{n}(\mathbb{R}^{\nambient}),$$
and exploit the metric $\distance$ to define a suitable adaptivity-score. We do this by relying on a generalization of the statistical variance, specifically designed for random variables %taking values 
in metric spaces, see, e.g., \cite{dubey2020functional} or Remark \ref{remark:variance} at the end of this Section,
\begin{equation}
\label{eq:dvariance}
\text{Var}(\dod):=\mathbb{E}_{\mub,\mub'}\left[\frac{1}{2}\distance^{2}(\tilde{\dod}_{\mub},\tilde{\dod}_{\mub'})\right],
\end{equation}
where $\mub,\mub'\sim\mathbb{P}$ are independent and identically distributed (recall that, here, $\mathbb{P}$ is a given probability distribution defined over $\pgeo$, modeling the likelihood of different parametric configurations). Since, by definition, we have $0\le\text{Var}(\dod)\le 1/2,$ we define the \textit{DOD adaptivity-score} as a normalized standard-deviation, namely
\begin{equation}
\label{eq:adapt-score}
\adapt(\dod):=\sqrt{2\text{Var}(\dod)}=\mathbb{E}^{1/2}_{\mub,\mub'}\left[\distance^{2}(\tilde{\dod}_{\mub},\tilde{\dod}_{\mub'})\right],
\end{equation}
so that $0\le \adapt(\dod)\le 1.$ The adaptivity-score has the following interpretation: values close to zero indicate that the DOD is collapsing towards a unique global basis; conversely, larger scores correspond to a larger variability across $\pgeo$.

From a practical point of view, since the DOD can be evaluated with little or no computational cost, we can estimate \eqref{eq:adapt-score} via classical Monte Carlo. Here, we shall rely on the following estimator
\begin{equation}
\label{eq:mc-adapt-score}
\adapt_{\textnormal{MC}}(\dod):=\sqrt{\frac{1}{N_{\text{r}}}\sum_{i=1}^{N_{\text{r}}}\distance^{2}(\tilde{\dod}_{\mub_{2i-1}^{\text{r}}}, \tilde{\dod}_{\mub_{2i}^{\text{r}}})},
\end{equation}
where $\{\mub_{i}^{\text{r}}\}_{i=1}^{2N_{\text{r}}}\subset\Theta$ are a suitable i.i.d. random sample, independent of the training set. In general, letting $N_{\text{r}}\sim10^4$ already yields a reasonable approximation, as demonstrated by the following.

\begin{lemma}
    Fix any $\delta,\epsilon>0$. Let $N_{\text{r}}=\left \lceil{\delta^{-1}\epsilon^{-4}}/4\right \rceil $ and let $\{\mub_{i}^{\text{r}}\}_{i=1}^{2N_{\text{r}}}\subset\Theta$ be an i.i.d. random sample, independent of the training set. Then,
    $$\left|\adapt(\dod)-\adapt_{\textnormal{MC}}(\dod)\right|\le\epsilon$$
    with probability $1-\delta.$
\end{lemma}
\begin{proof}
    Let $\textsf{Prob}$ denote the probability law of the entire random sample. 
    We note that, for all $a,b\ge0$, one has $|a^{2}-b^{2}|\ge|a-b|^{2}.$ Therefore,
    \begin{multline*}        
        \textsf{Prob}\left(\left|\adapt(\dod)-\adapt_{\textnormal{MC}}(\dod)\right|>\epsilon\right)\\\le\textsf{Prob}\left(\left|\adapt^{2}(\dod)-\adapt_{\textnormal{MC}}^{2}(\dod)\right|>\epsilon^{2}\right)\\\le\epsilon^{-4}\text{Var}(\adapt_{\textnormal{MC}}^{2}(\dod)),
    \end{multline*}
    by Chebyshev inequality. Let $X:=\distance^{2}(\tilde{\dod}_{\mub^{\text{r}}_{1}},\tilde{\dod}_{\mub^{\text{r}}_{2}})$. By independence and since both $X$ and $\mathbb{E}[X]$ take values in $[0,1]$, we have
    \begin{multline*}
    \text{Var}(\adapt_{\textnormal{MC}}^{2}(\dod))=\frac{1}{N_{\text{r}}}\text{Var}(X)=\frac{1}{N_{\text{r}}}\left(\mathbb{E}[X^{2}]-\mathbb{E}^{2}[X]\right)\le\\\le\frac{1}{N_{\text{r}}}\left(\mathbb{E}[X]-\mathbb{E}^{2}[X]\right)=\frac{1}{N_{\text{r}}}\mathbb{E}[X]\left(1-\mathbb{E}[X]\right)\le\frac{1}{4N_{\text{r}}}.   \end{multline*}
    Consequently, $\textsf{Prob}\left(\left|\adapt(\dod)-\adapt_{\textnormal{MC}}(\dod)\right|>\epsilon\right)\le (4N_{\text{r}})^{-1}\epsilon^{-4}\le\delta.$
\end{proof}

\begin{remark}
    \label{remark:variance}
    Equation \eqref{eq:dvariance} can be seen as a generalization of the variance to metric spaces. To see this, consider the case of a real-valued random variable $X$. Classically, its variance is defined as $\textnormal{Var}(X)=\mathbb{E}[|X-\mathbb{E}[X]|^{2}]=\mathbb{E}[X^{2}]-\mathbb{E}^{2}[X].$ We now note that if $Y$ is another independent random variable and $Y\sim X$ (identical distribution), then by classical properties,
    \begin{multline*}
        \mathbb{E}\left[\frac{1}{2}|X-Y|^{2}\right]=\frac{1}{2}\left(\mathbb{E}[X^{2}]-2\mathbb{E}[XY]+\mathbb{E}[Y^{2}]\right)=\\=\frac{1}{2}\left(\mathbb{E}[X^{2}]-2\mathbb{E}[X]\mathbb{E}[Y]+\mathbb{E}[Y^{2}]\right)=\\=\frac{1}{2}\left(\mathbb{E}[X^{2}]-2\mathbb{E}[X]\mathbb{E}[X]+\mathbb{E}[X^{2}]\right)=\textnormal{Var}(X).
    \end{multline*}
    In particular, if we denote the Euclidean distance by $\distance=|\cdot|$, then the above is $\mathbb{E}\left[\frac{1}{2}\distance(X,Y)^{2}\right]=\textnormal{Var}(X)$, thus motivating our definition in \eqref{eq:dvariance}.
\end{remark}

\begin{remark}
    \label{remark:extreme-values}
    It is worth pointing out that both the case $\adapt(\dod)=0$ and $\adapt(\dod)=1$, are somewhat pathological. In the former case, in fact, the DOD basis is constant, which is equivalent to a classical POD. In the latter case, instead, the DOD turns out to be discontinuous. To see this, note that the continuity of the DOD implies that of the map $g:\mub_{1},\mub_{2}\mapsto \distance^{2}(\dod_{\mub_{1}},\dod_{\mub_{2}})$. Assume now that $\mub$ is an absolutely continuous random variable whose density never vanishes over $\pgeo$. Then $\adapt(\dod)=1$ would imply $g=1$ almost everywhere. If $\dod$ were to be continuous, this would imply $g\equiv1$ over $\pgeo\times\pgeo$; however, this is not possible, since $g(\mub_{1},\mub_{1})=0$ for all $\mub_{1}\in\pgeo$. Thus, $\adapt(\dod)=1$ can only be achieved on a discontinuously adaptive basis.
\end{remark}

\section{Numerical experiments: dimensionality reduction}
\label{sec:exp1}
The purpose of this Section is to provide some preliminary insights on the capabilities of the DOD algorithm as a tool for dimensionality reduction. To this end, we shall present a couple of numerical experiments in which we compare the DOD with other well-established approaches, specifically: POD, clustered POD and autoencoders. 

Similarly to DOD, all these techniques are data-driven, meaning that they require the preliminary collection of some FOM snapshots, $\{\mub_{i},\nub_{i},\ub_{\mub_{i},\nub_{i}}\}_{i=1}^{N}$, randomly sampled, which serve as training data. 
All these approaches define a latent space, where solutions are projected (linearly or nonlinearly) and from which they can be later recovered. To evaluate the quality of the reconstruction, we rely on the mean relative projection error (MRPE)
$$\textnormal{MRPE}:=\frac{1}{N_{\textnormal{test}}}\sum_{i=1}^{N_{\textnormal{test}}}\frac{\|\ub_{\check{\mub}_{i},\check{\nub}_{i}}-\ub_{\check{\mub}_{i},\check{\nub}_{i}}^{\textnormal{proj}}\|}{\|\ub_{\check{\mub_{i}},\check{\nub_{i}}}\|},$$
where $\|\cdot\|$ is the norm induced by $\mass$ over $\mathbb{R}^{N_{h}}$, corresponding to the $L^{2}$-norm in $V_{h}$, while $\{\check{\mub}_{i},\check{\nub}_{i},\ub_{\check{\mub}_{i},\check{\nub}_{i}}\}_{i=1}^{N_{\textnormal{test}}}$ is the so-called test set, a collection of high quality data generated independently of the training set. Here, depending on the dimensionality reduction technique, $\ub_{\mub_{i},\nub_{i}}^{\text{proj}}$ denotes the reconstruction of $\ub_{\mub_{i},\nub_{i}}$. In the DOD case, this is simply given by
$$\ub_{\mub,\nub}^{\textnormal{proj}}:=\dod_{\mub}\dod_{\mub}^{\top}\mass\ub_{\mub,\nub}.$$
In the other cases, instead, the formulas are slightly different. We report them below.

\begin{itemize}
    \item \textbf{POD} \cite{quarteroni2016reduced}. In this case, 
    $$\ufomp^{\text{proj}}:=\mathbb{V}\mathbb{V}^{\top}\mass\ufomp$$
    where, for a given reduced dimension $n$, $\mathbb{V}\in\mathbb{R}^{\fomdim\times n}$ represents a global basis computed by generalized SVD, as in Algorithm \ref{algo:ambient};
    
    \item \textbf{Clustered POD} \cite{pagani2018numerical, geelen2022localized}. Here, the reconstruction reads
    $$\begin{cases}
        \ufomp^{\text{proj}}:=\mathbb{V}_{j}\mathbb{V}_{j}^{\top}\mass\ufomp,\\
        j=\argmin_{k=1,\dots, c}\;\|\ufomp-\mathbb{V}_{k}\mathbb{V}_{k}^{\top}\mass\ufomp\|,
    \end{cases}$$
    where $c$ is the number of clusters, $\mathbb{V}_{1},\dots\mathbb{V}_{c}\in\mathbb{R}^{\fomdim\times n}$ is a collection of basis and $n$ is the reduced dimension. In practice, given $c$ and $n$, the FOM data are first subdivided into $c$ clusters by grouping together similar solutions (here, we rely on the $k$-means algorithm); then, a POD basis is computed for each cluster, yielding the matrices $\mathbb{V}_{1},\dots\mathbb{V}_{c}$. Then, each solution is projected and reconstructed using its own POD basis, defined as the "best" among the ones available. It can be regarded as a primitive form of DOD, where the basis is piecewise constant over the parameter space and changes discontinuously. Typically, methods known as "dictionary-based ROMs" tend to rely on this approach for their construction \cite{daniel2020model, herkert2024dictionary};

    \item \textbf{Autoencoders} \cite{franco2023deep, fresca2021comprehensive, romor2023non}. In this case, the formula is just
    $$
    \ufomp^{\text{proj}}:=\Psi(\Psi'(\ufomp)),
        $$
    where $\Psi':\mathbb{R}^{\fomdim}\to\mathbb{R}^{n}$ and $\Psi:\mathbb{R}^{n}\to\mathbb{R}^{\fomdim}$ are the encoder and decoder networks, respectively. To foster interpretability and provide a meaningful comparison, we shall construct these models following the same ideas adopted for the design of DOD architectures: %rules of thumb adopted for the design of DOD architectures. In particular, we shall consider autoencoders that are comparable in terms of complexity (total number of weights and biases to be optimized during training) and design (layer types, depth, etc.). Finally, as for the DOD, 
    in particular, we shall rely on POD enhanced autoencoders \cite{brivio2023error, fresca2021pod}, thus leveraging over the existence of the ambient space $\ambient$. In other words, we let $$\Psi'(\ub)=\psi'(\ambient^{\top}\ub)\;\;\textnormal{and}\;\;\Psi(\mathbf{c})=\ambient\psi(\mathbf{c}),$$ where $\psi':\mathbb{R}^{\nambient}\to\mathbb{R}^{n}$ and $\psi:\mathbb{R}^{n}\to\mathbb{R}^{\nambient}$ are the trainable parts of the two architectures, respectively.
\end{itemize}

We conduct the analysis as follows. First, we compare DOD, POD, and AE for varying $n$, so as to better understand how the reduced dimension impacts the projection error. Then, to see whether the clustered POD algorithm can replicate the results achieved by the DOD projector, we compare the two for a fixed dimension $n$ and a varying number of POD clusters $c$.
\\\\
All the code was implemented in Python 3 using the \textit{dlroms} library, a comprehensive Python package that exploits FEniCS and Pytorch to construct deep learning-based ROMs. The \textit{dlroms} package can be accessed for free on Github: \url{https://github.com/NicolaRFranco/dlroms}. 

%\wip

\begin{figure}
    a)\hspace{6.5cm}b)
    \begin{center}
    \includegraphics[width=0.42\textwidth]{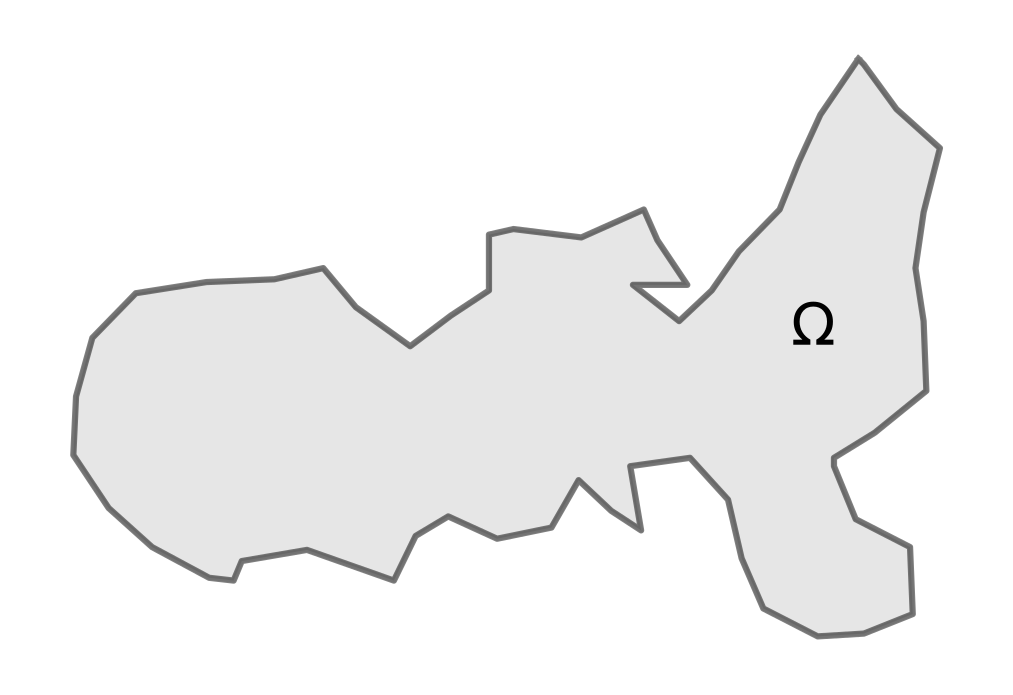}\hfill
    \includegraphics[width=0.5\textwidth]{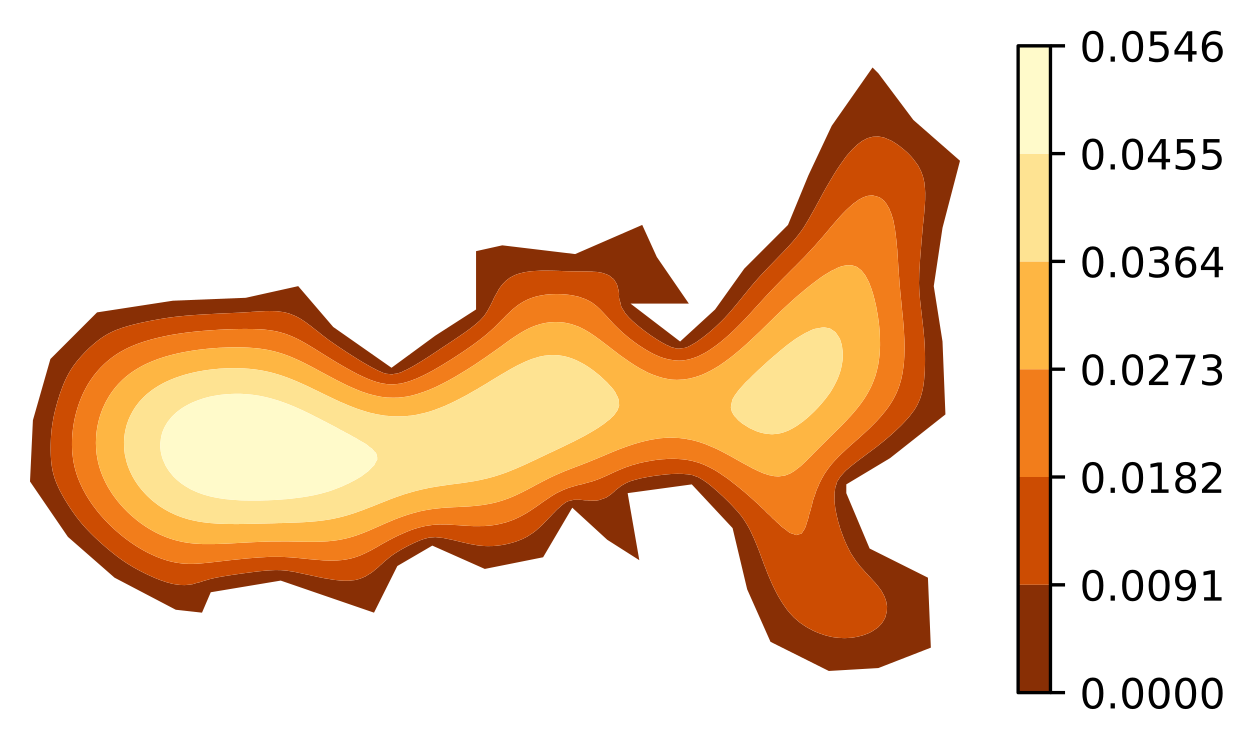}
    \caption{a) spatial domain for the Eikonal equation example, Section \ref{subsec:eikonal}; b) distance to boundary map.}
    \label{fig:domain-elba}
    \end{center}
\end{figure}

\subsection{Eikonal equation}
\label{subsec:eikonal}
To showcase the ability of the DOD to handle arbitrary scenarios, we start by considering a nonlinear PDE defined over a complicated domain. In particular, let
$\Omega$ be the spatial domain in Figure \ref{fig:domain-elba}, representing a simplified cartography of the Elba island, in Italy. We consider a parameter dependent Eikonal equation,
\begin{equation}
\label{eq:eikonal}
|\nabla u|= s_{\nub}^{-1},
\end{equation}
complemented with an internal Dirichlet condition, $u(\xb_{\mub})=0.$ In literature, this equation constitutes a prototypical example of wave propagation: for instance, in seismology and geophysics, it is commonly employed for modeling travelling times of seismic waves through the Earth's subsurface \cite{lin2009eikonal, ma2014calculating}.

For our analysis, we consider a situation in which  \eqref{eq:eikonal} %formally 
depends on four scalar parameters: $\nub=[\nu_{1},\nu_{2}]$ and $\mub=[\mu_{1},\mu_{2}]$. The former, which take values in $\pphys:=[0.1, 30]\times[0.001, 0.01],$ are used to parametrize the speed of travel $s_{\nub}:\Omega\to(0,+\infty)$ as
$$s_{\nub}:= \nu_{1}\left(\|d\|_{L^{\infty}(\Omega)}-d\right)+\nu_{2},$$
where $d:\Omega\to[0,+\infty)$ is the distance to boundary map, hereby computed by solving a preliminary Eikonal equation of the form: $|\nabla d|=1$ in $\Omega$, $d\equiv0$ on $\partial\Omega$. We refer to Figure \ref{fig:domain-elba} (right panel), for a visual depiction of $d$. 

The other set of parameters $\mub$, instead, is used to parameterize the location of a given source $\xb_{\mub}$. To this end, we simply let $\mu_{1}$ and $\mu_{2}$ be the coordinates of $\xb_{\mub},$ so that, ultimately, $\xb_{\mub}=\mub$ and $\Theta:=\Omega$.
From a physical point of view, given $\nub\in\pphys$ and $\mub\in\pgeo$, the solution $u=u_{\mub,\nub}(\xb)$ represents the travel time from $\xb_{\mub}$ to $\xb$, under the speed limit $s_{\nub}.$ Intuitively, since $\mub$ enters the PDE in a singular way, directly interacting with the space variable $\mathbf{x}$, we expect a $\mub$-adaptive approach, such as the DOD, to perform better than classical techniques adopting a global perspective.

As ground truth reference, we consider a stabilized FOM (artificial diffusion $\epsilon=0.1$) based off a Finite Element discretization of \eqref{eq:eikonal} via continuous P1 elements defined over a triangular mesh of stepsize $h\approx 0.0084$, resulting in a total of $N_{h}=9550$ degrees of freedom (dof). We exploit the FOM to sample 1000 random solutions, 500 for training, and 500 for testing. 

To construct the DOD projector, we rely on an ambient space of dimension $N_{A}=300$, as empirically enough to capture most of the variability in the solution manifold (average ambient error = 0.31\%). The remaining parts of the architecture are as in Table \ref{tab:elba - dod architecture}. To ensure a proper comparison, the autoencoders are constructed similarly (see \ref{sec:appendix:architectures} for further details).
The results are shown in Figs. \ref{fig:eikonal-decay}-\ref{fig:elba-basis}, Fig. \ref{fig:adaptivity}, and Table \ref{tab:dod performances}.

\begin{table}
    \centering
    \begin{tabular}{lll}
    \hline\hline
        \textbf{Component} &  \textbf{Specifics} & \textbf{Terminal activation}\\\hline
        Seed & $p\textcolor{white}{p}\mapsto500\mapsto 50$ & 0.1-leakyReLU\\
        Root & $50\mapsto100\mapsto N_{A}$& - \\
        Orth & reduced QR & -\\\hline\hline
    \end{tabular}
    \caption{General DOD architecture for the Eikonal Equation case study, Section \ref{subsec:eikonal}. All architectures employ the 0.1-leakyReLU activation at the \textit{internal} layers. The notation $a\mapsto b$ denotes a dense layer from $\mathbb{R}^{a}$ to $\mathbb{R}^{b}$; longer sequences indicate a composition of multiple layers. The number of root modules depends on the DOD dimension, $n$.}
    \label{tab:elba - dod architecture}
\end{table}

\;\\As we can appreciate from Figure \ref{fig:eikonal-decay} (left panel), for any fixed latent dimension $n$, the DOD approach emerges by far as the best dimensionality reduction technique, reporting errors that are 5 to 10 times smaller than those achieved by POD and autoencoders. In turn, this results in a significant gain in terms of compression rate: note, for instance, that 4 DOD modes can provide the same information as 30 POD modes.

As illustrated in Figure \ref{fig:elba-basis}, the key factor that enables all of this is the adaptability of the DOD basis. There, we see how the DOD modes change according to the position of the geometrical parameter $\mathbf{x}_{0}\in\Omega$. Here, for the sake of readability, we report the results obtained for $n=3$. It is interesting to see how, although different, the basis functions appear to follow a specific criterion. In fact, the first mode is mostly active on the east side of the domain, while the second and the third one capture the west coast and the center region, respectively. In this sense, it appears that the DOD can automatically learn some form of domain decomposition. It should be noted that these considerations are only possible because of the high interpretability that is intrinsic to the DOD approach, something that, in contrast, autoencoders can hardly achieve.
\\\\
Of note, the adaptivity of the DOD basis is not easy to replicate if one simply relies on multiple local basis, as, e.g., in the clustered POD algorithm. This becomes apparent when examining Figure \ref{fig:eikonal-decay} (right panel), where it is evident that a substantial number of POD clusters are needed to replicate the performance of the DOD. For example, extrapolating from the overall trend indicates that approximately $c\approx1900$ clusters would be required to match the accuracy of the DOD basis, a conclusion that is clearly non-sensical.

At the same time, higher variability does not necessarily imply a better accuracy. For example, when going from $n=3$ to $n=4$, we see that the DOD basis becomes less volatile but more accurate; see Figures \ref{fig:eikonal-decay} and \ref{fig:adaptivity}. As we shall discuss in Section \ref{sec:dodrom}, this fact should be taken into account when constructing a DOD based reduced order model: in fact, since the DOD coefficients are computed as $\mathbf{c}_{\mub,\nub}=\dod_{\mub}^{\top}\mass\ufomp$, a larger variability in the DOD basis can produce a higher volatility in the DOD coefficients, making them harder to learn.

Last but not least, we devote a final comment to the error trends in Figure \ref{fig:eikonal-decay}, which at first may seem counterintuitive. In fact, one might expect the DOD error to decay faster than the POD (resp. AE) projection error, but Figure \ref{fig:eikonal-decay} seems to suggest otherwise.
However, this phenomenon is easily explained. If we look closer, we see that from $n=2$ to $n=4$, the error decay is much faster, thus resembling the expected behavior. However, it is also clear that this trend cannot continue for larger $n$. In fact, the accuracy of the DOD projector is bounded by that of the ambient space, in the sense that, for every $n$, one has
$$\mathbb{E}\left[\frac{\|\ufomp-\dod_{\mub}\dod_{\mub}^{\top}\mass\ufomp\|}{\|\ufomp\|}\right]\ge\mathbb{E}\left[\frac{\|\ufomp-\ambient\ambient^{\top}\mass\ufomp\|}{\|\ufomp\|}\right]=0.31\%.$$
Since the projection error is already 0.97\% for $n=4$, cf. Table \ref{tab:dod performances}, it is evident that the error decay will eventually flatten out, leading to the behavior in Figure \ref{fig:eikonal-decay}.

\begin{figure}
    \begin{center}
    \includegraphics[width=0.495\textwidth]{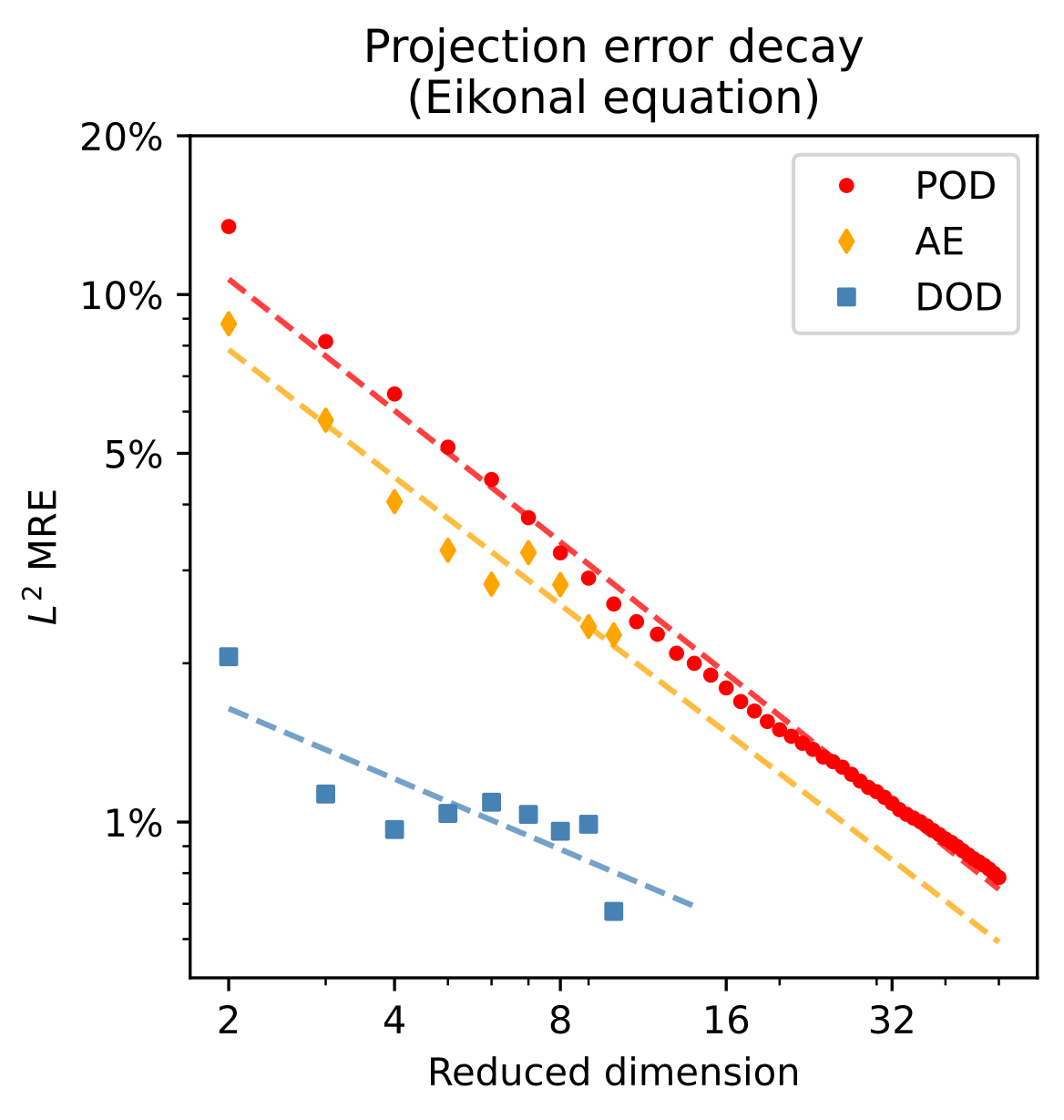}\hfill
    \includegraphics[width=0.495\textwidth]{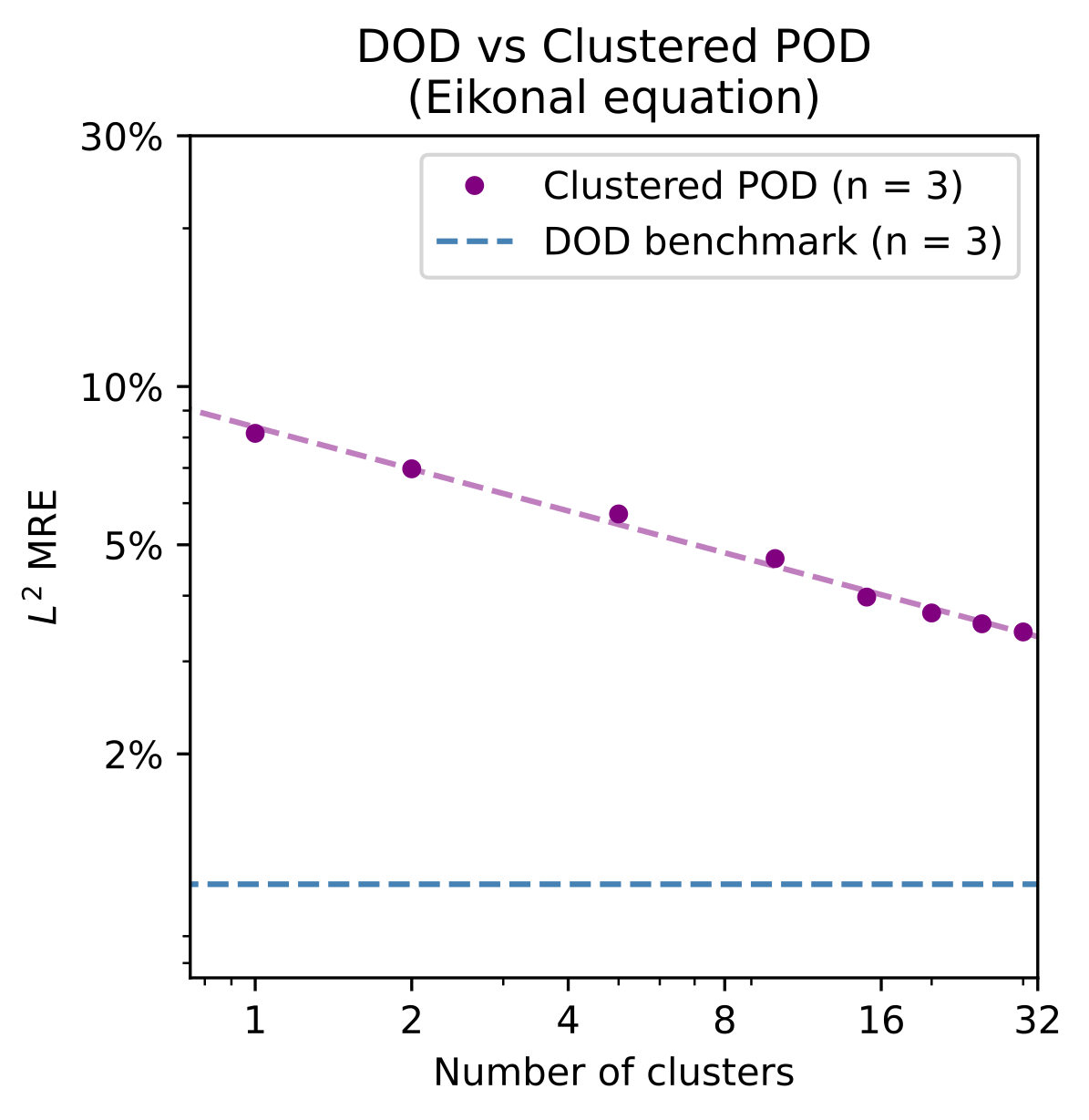}
    \caption{Comparison between DOD and other dimensionality reduction strategies for the Eikonal equation example, Section \ref{subsec:eikonal}.}
    \label{fig:eikonal-decay}
    \end{center}
\end{figure}

\begin{figure}
    \centering
    DOD mode \#1\hspace{2.6cm}
    DOD mode \#2\hspace{2.6cm}
    DOD mode \#3\;\;\;\;\;\;\;\;
    \includegraphics[width=\textwidth]{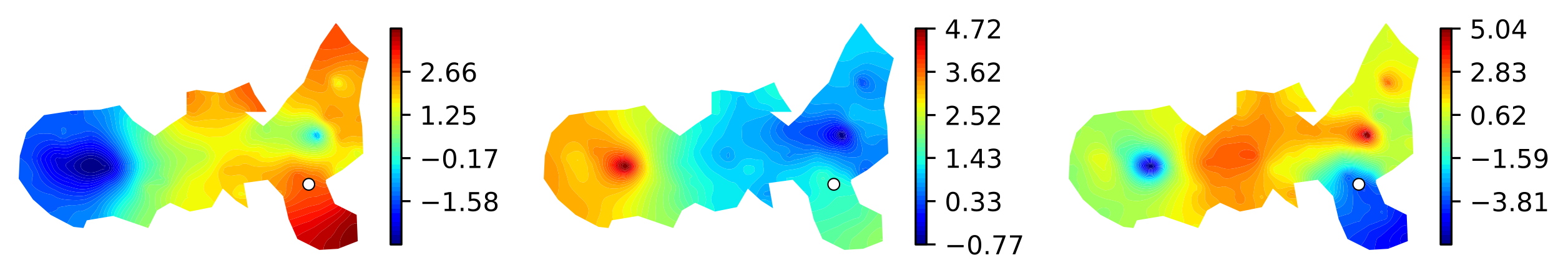}
    \includegraphics[width=\textwidth]{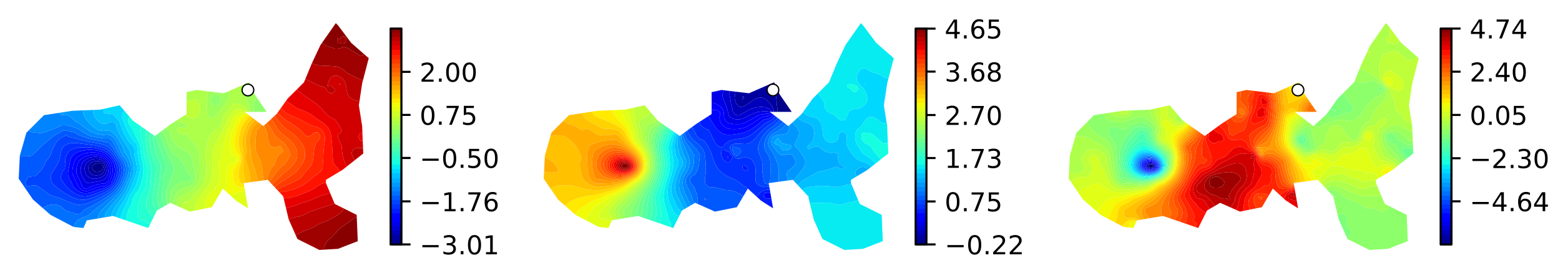}
    \includegraphics[width=\textwidth]{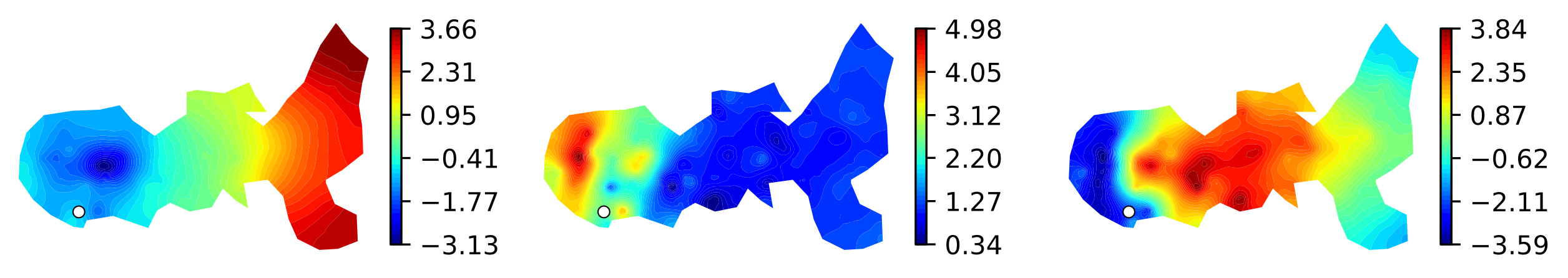}
    \caption{DOD basis for varying positions of the geometrical parameter $\mub=\xb_{0}$ (white dot) in the Eikonal equation example, Section \ref{subsec:eikonal}. Each row refers to a different value of $\mub$, while each column represents a DOD mode (here, $n=3$).}
    \label{fig:elba-basis}
\end{figure}

\subsection{Stationary Navier-Stokes flow around a parametrized obstacle}
\label{subsec:navier-stokes}
For our second example, we return to our model problem, first discussed in Section \ref{subsec:example}, concerning a steady fluid flow around an obstacle. For better readability, we take the opportunity to restate the problem, specifying the governing equations and their parameterization. 

We consider a 2D fluid flow modeled by the following parametrized Navier-Stokes equations, 

\begin{equation}
    \label{eq:navier-stokes}
    \begin{cases}
    -\epsilon\Delta \ubc + \ubc\cdot\nabla\ubc + \nabla q = 0 & \text{in}\;\Omega_{\mub},\\
    \nabla\cdot\ubc = 0 & \text{in}\;\Omega_{\mub},\\
    \ubc = \mathbf{g}_{\nub} & \text{on}\;\Gamma_{\text{in}},\\
    \ubc = 0 & \text{on}\;\partial\Omega_{\mub}\setminus\left(\Gamma_{\text{in}}\cup\Gamma_{\text{out}}\right),\\
    q = 0 & \text{on}\;\Gamma_{\text{out}}
\end{cases}
\end{equation}
where $\Omega_{\mub}:=(0,1)^{2}\setminus O_{\mub}$ is a parameter dependent domain, obtained by removing an almond-shaped object, $O_{\mub}$, from the unit square, see Fig. \ref{fig:navier-stokes-domain}. We focus our attention on the parameters-to-velocity map,
$$(\mub,\nub)\mapsto \ubc.$$
Here, $\mub=[\theta,x_{0},y_{0}]$ is a vector parametrizing the center of the obstacle, $(x_{0},y_{0})$, and its angle of rotation, $\theta$. To ensure that the obstacle $O_{\mub}$ always lies within the unit square, we let $\mub\in\pgeo:=[0,2\pi]\times[0.25, 0.75]^{2}$.

The other parameters, $\nub=[\alpha,\beta]\in\pphys:=[0, 10]^{2}$, instead, parametrize the inflow condition as
$$\mathbf{g}_{\nub}(x,y)=y(1-y)\left(\alpha e^{-100(y-0.25)^{2}}+\beta e^{-100(y-0.75)^{2}}\right)^{1/2}.$$
Simply put, the inflow $\mathbf{g}_{\nub}$ consists of two main contributions: one coming from the bottom (centered at $y=0.25)$, whose strength is determined by $\alpha$, and one coming from the top ($y=0.75)$, whose intensity depends on $\beta$. For simplicity, the viscosity coefficient is fixed instead to $\epsilon:=5\cdot10^{-3}.$

In general, since the parameters directly affect the geometry of the problem, constructing a ROM for \eqref{eq:navier-stokes} can be highly non-trivial, as most approaches require FOM solutions to belong to the same discrete functional space (see also Remark \ref{remark:geometry} at the end of this Section). To account for this, we shall consider a FOM based on a fictitious domain approach, where \eqref{eq:navier-stokes} is embedded in the enlarged domain $\Omega_{\text{e}}:=(0,1)^{2}\supset\Omega_{\mub}$, and additional Dirichlet conditions are imposed to adjust for the location of the obstacle. We discretize the fictitious domain $\Omega_{\text{e}}$ with a structured triangular grid $50\times50$ and approximate the solution to \eqref{eq:navier-stokes} using a standard approach based on Picard iterations (tollerance =1e-10). In order to represent the discretized pressure and velocity fields, we use a stable pairing based on mini-elements (continuous P1 elements for $q$, and P1-Bubble vector elements $\ubc$). 
Then, the FOM solver defines a map of the form
$$(\mub,\nub)\mapsto\mathbf{u}_{\mub,\nub}\in\mathbb{R}^{N_{h}}$$
where $N_{h}=15202$ are the dof in the P1-Bubble space of vector fields, $V_{h}\subset H^{1}(\Omega) \times H^{1}(\Omega).$ 
We exploit the FOM to sample 1500 random solutions, 1000 for training and 500 for testing.
As before, we construct the DOD projector using an ambient space of dimension $N_{A}=300$ (average ambient error = 0.89\%), and a collection of dense architectures for the seed and roots modules, see Table \ref{tab:nstokes - dod architecture}. The results are shown in Figs. \ref{fig:navier-stokes-decay}-\ref{fig:nstokes-basis}, Fig. \ref{fig:adaptivity} and Table \ref{tab:dod performances}.

\begin{table}
    \centering
    \begin{tabular}{lll}
    \hline\hline
        \textbf{Component} &  \textbf{Specifics} & \textbf{Terminal activation}\\\hline
        Seed & $p\textcolor{white}{p}\stackrel{*}{\mapsto}4\textcolor{white}{4}\mapsto 50$ & 0.1-leakyReLU\\
        Root & $50\mapsto50\mapsto N_{A}$& - \\
        Orth & reduced QR & -\\\hline\hline
    \end{tabular}
    \caption{General DOD architecture for the Navier-Stokes case study, Section \ref{subsec:navier-stokes}. The table entries read as in Table \ref{tab:elba - dod architecture}. All architectures employ the 0.1-leakyReLU activation at the \textit{internal} layers. Here, $\stackrel{*}{\mapsto}$ denotes a non-learnable feature layer that acts as $[\theta, x_{0},y_{0}]\stackrel{*}{\mapsto}[\cos4\theta,\sin4\theta,x_{0},y_{0}]$, which we use to enforce rotational symmetry.}
    \label{tab:nstokes - dod architecture}
\end{table}

\;\\
In general, the results are very similar to those obtained in the previous case study, Section \ref{subsec:eikonal}, at least from a qualitative point of view. For fixed $n$, the performance achieved by the DOD algorithm is unmatched, with POD and autoencoders being twice as bad in compressing information. However, this gap tends to decrease for larger $n$, as the projection error of the DOD seems to decay slower: as in the previous case study, in the long run, this might be due to the underlying ambient space $\spann(\ambient)$. %As before, the error decay appears to slow down after $n=4$, possibly due to the underlying ambient space $\spann(\ambient).$

As the reduced dimension increases, we also observe an increased volatility of the DOD basis, with adaptivity scores ranging from 0.5 to 0.9, cf. Figure \ref{fig:adaptivity}. Adaptivity is also evident in Figure \ref{fig:nstokes-basis}, where we can clearly appreciate how the DOD network is capable of crafting a specific modal basis for each obstacle configuration. Interestingly, we also note that each DOD mode focuses on different features of the problem. For instance, %we can observe that 
the third mode distinctly depicts the interaction between the flow and the obstacle, whereas the other two seem to model the interplay between the incoming jet-flows. %two modes illustrate the interplay between the two incoming inflows. 

This shows that, even for very small latent dimensions, $n=3$, the DOD approach can provide very rich representations. On the contrary, dictionary-based approaches, such as clustered POD, fail in replicating such complexity, unless the number of clusters becomes extremely large (Fig. \ref{fig:navier-stokes-decay}, right panel). In this sense, it appears that the main strength of the DOD approach is that of relying on a \textit{continuously adaptive} local basis.

\begin{figure}[h!]
    \centering
    DOD mode \#1\hspace{2.5cm}
    DOD mode \#2\hspace{2.5cm}
    DOD mode \#3\;\;\;\;\;\;\;\;
    \includegraphics[width=\textwidth]{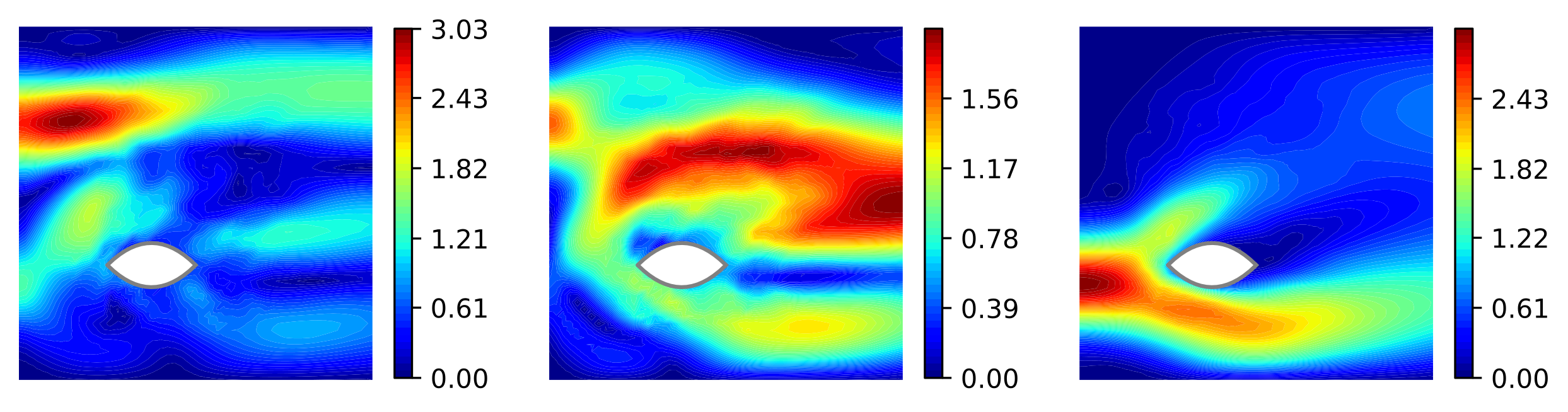}
    \includegraphics[width=\textwidth]{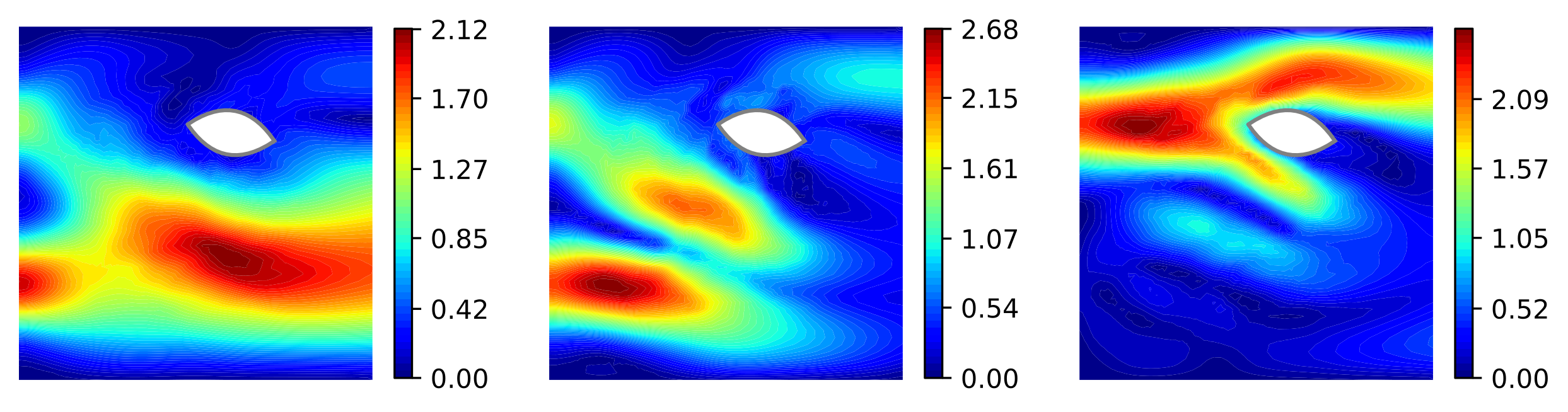}
    \includegraphics[width=\textwidth]{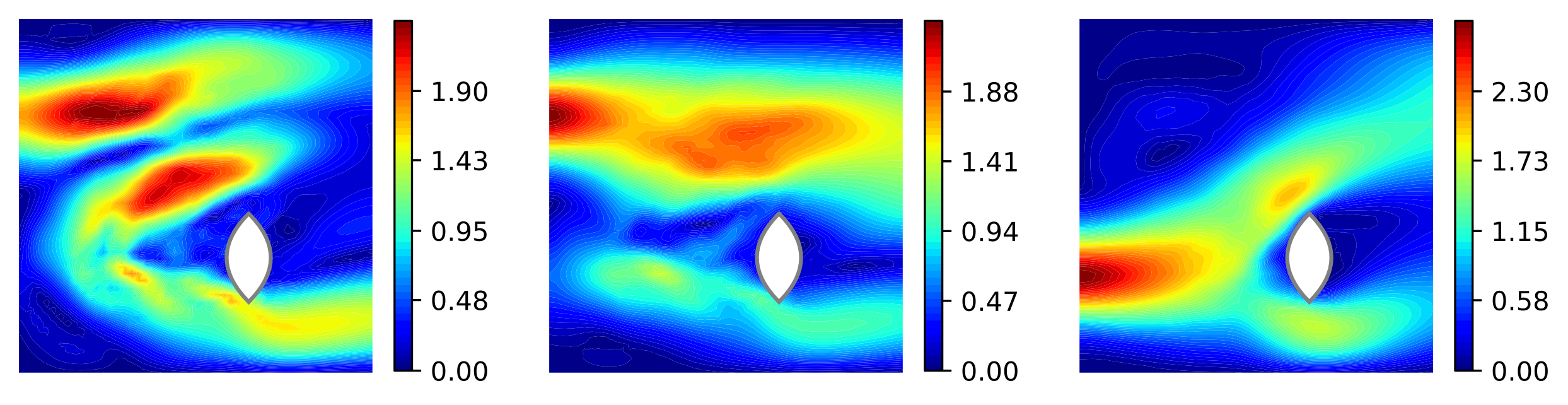}
    \caption{DOD basis for different positions (and rotations) of the obstacle in the Navier-Stokes example, Section \ref{subsec:navier-stokes}. Each row refers to a different value of $\mub=[\theta,x_{0},y_{0}]$, while each column represents a DOD mode (here, $n=3$). NB: for this case study, each DOD mode is actually a vector field $\mathbf{v}_{\mub}^{j}\Omega_{e}\to\mathbb{R}^{2}$. However, to enhance readability, we are only plotting their magnitudes, $|\mathbf{v}_{\mub}^{j}|$.}
    \label{fig:nstokes-basis}
\end{figure}

\begin{remark}
    \label{remark:geometry}
    Developing ROMs to tackle problems in varying geometries is a very challenging task. Possible strategies to address this task typically consist of: (i) relying on a fictitious domain approach, or on a suitable postprocessing routine that interpolates all PDE solutions over a common mesh \cite{bourguet2011reduced, liberge2010reduced}, (ii) exploiting mesh deformation strategies \cite{antil2014application, yin2024dimon}, %, with its inherent limitations (the geometries must be diffeomorphic and sufficiently similar to avoid singularities during deformation) \cite{},
    and/or registration methods \cite{taddei2020registration},
    (iii) leveraging on local operations, as in graph neural networks (GNNs) \cite{barwey2023multiscale, franco2023gnn, gladstone2023gnn}. Here, we consider the simplest of these approaches (that is, the first one), in order to maintain our focus on our primary objective, i.e., developing an adaptive local basis capable of overcoming the Kolmogorov barrier. Clearly, integrating DOD with, e.g., GNNs, would be an interesting research direction, potentially leading to very powerful and flexible ROMs. However, given that the DOD approach is still at its infancy, we leave these considerations for future work.
\end{remark}

\begin{figure}[ht!]
    \begin{center}
    \includegraphics[width=0.495\textwidth]{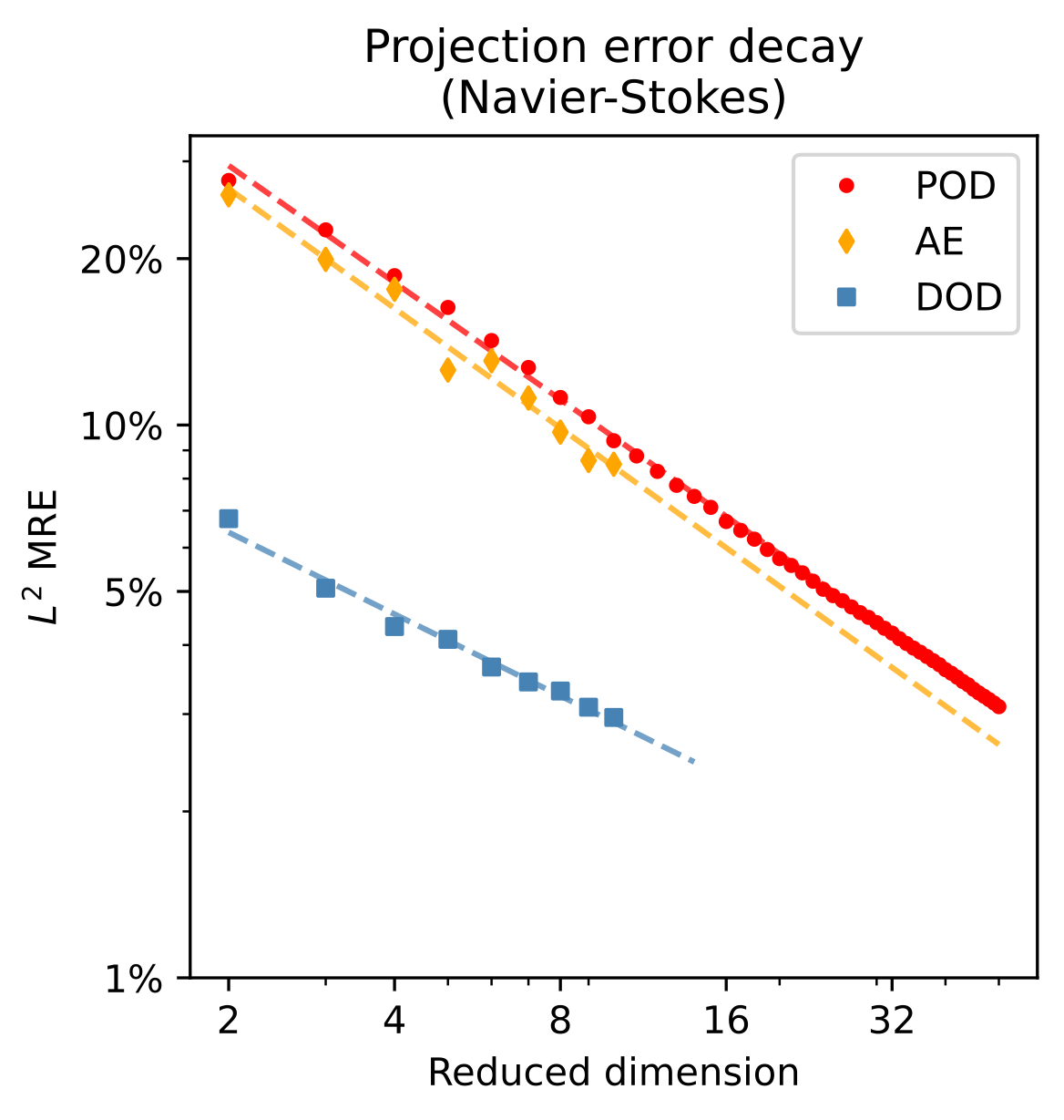}\hfill
    \includegraphics[width=0.495\textwidth]{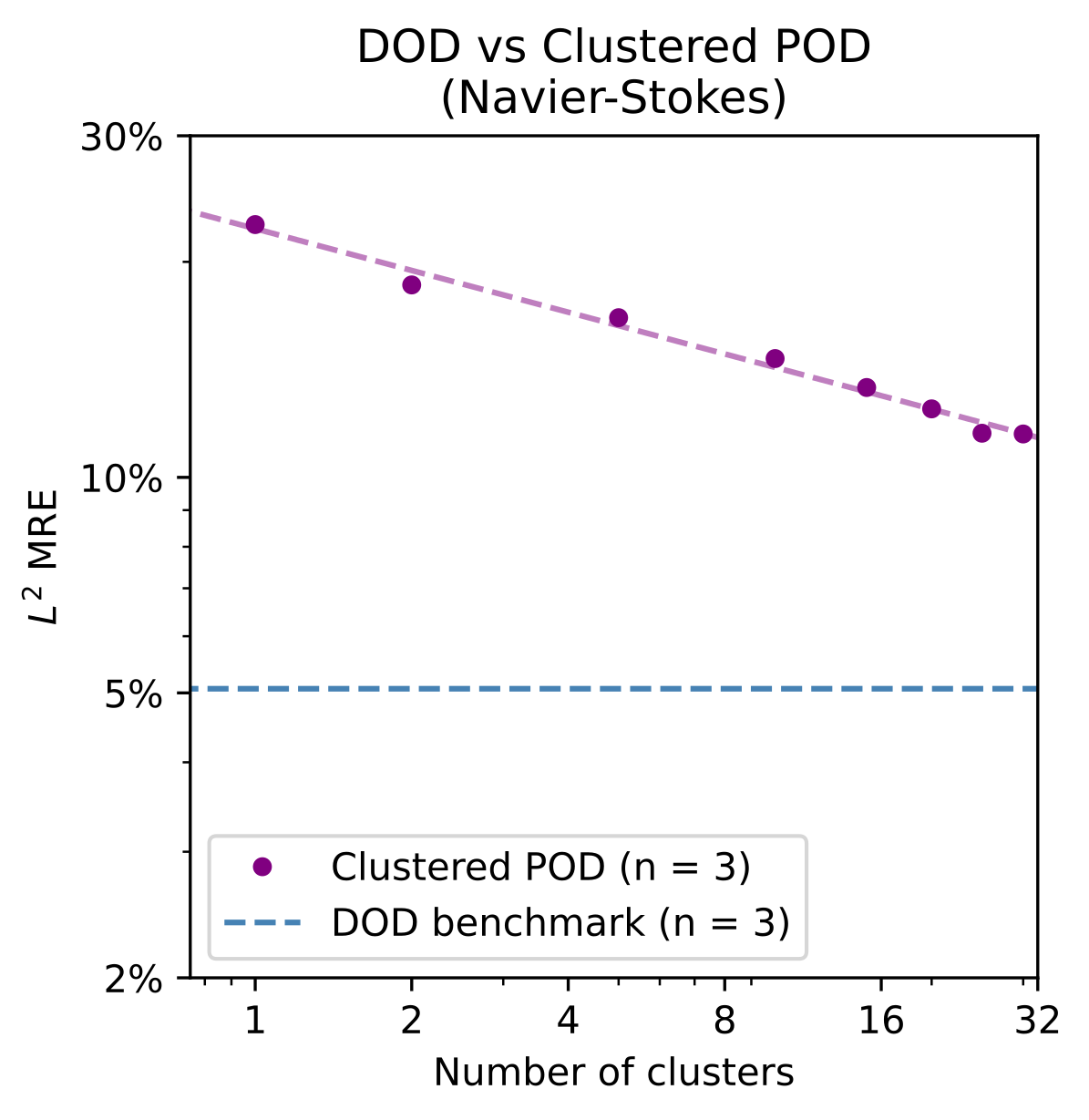}
    \caption{Comparison between DOD and other dimensionality reduction strategies for the Navier-Stokes example, Section \ref{subsec:navier-stokes}.}
    \label{fig:navier-stokes-decay}
    \end{center}
\end{figure}

\begin{figure}[ht!]
    \centering
    \includegraphics[width = 0.4\textwidth]{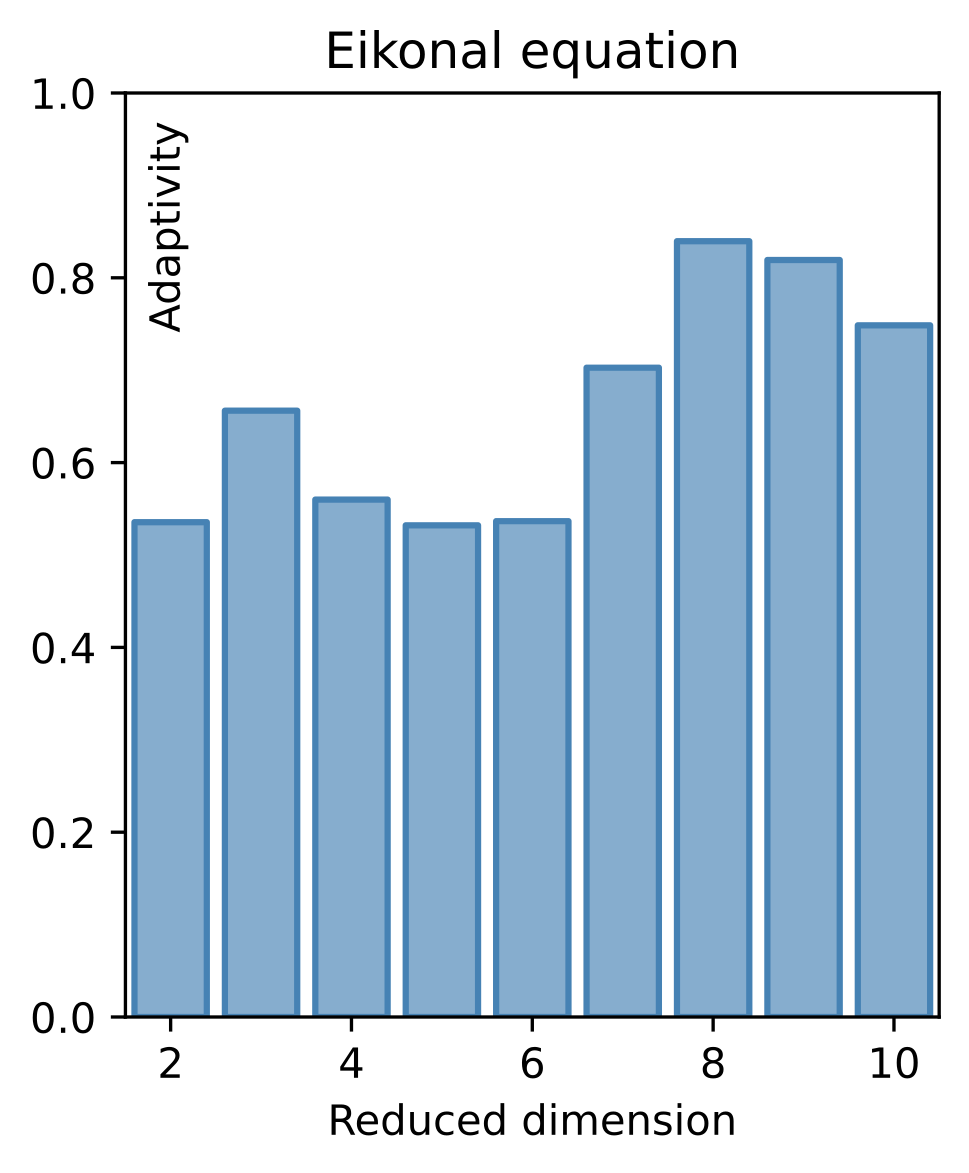}
    \hspace{1.5cm}
    \includegraphics[width = 0.4\textwidth]{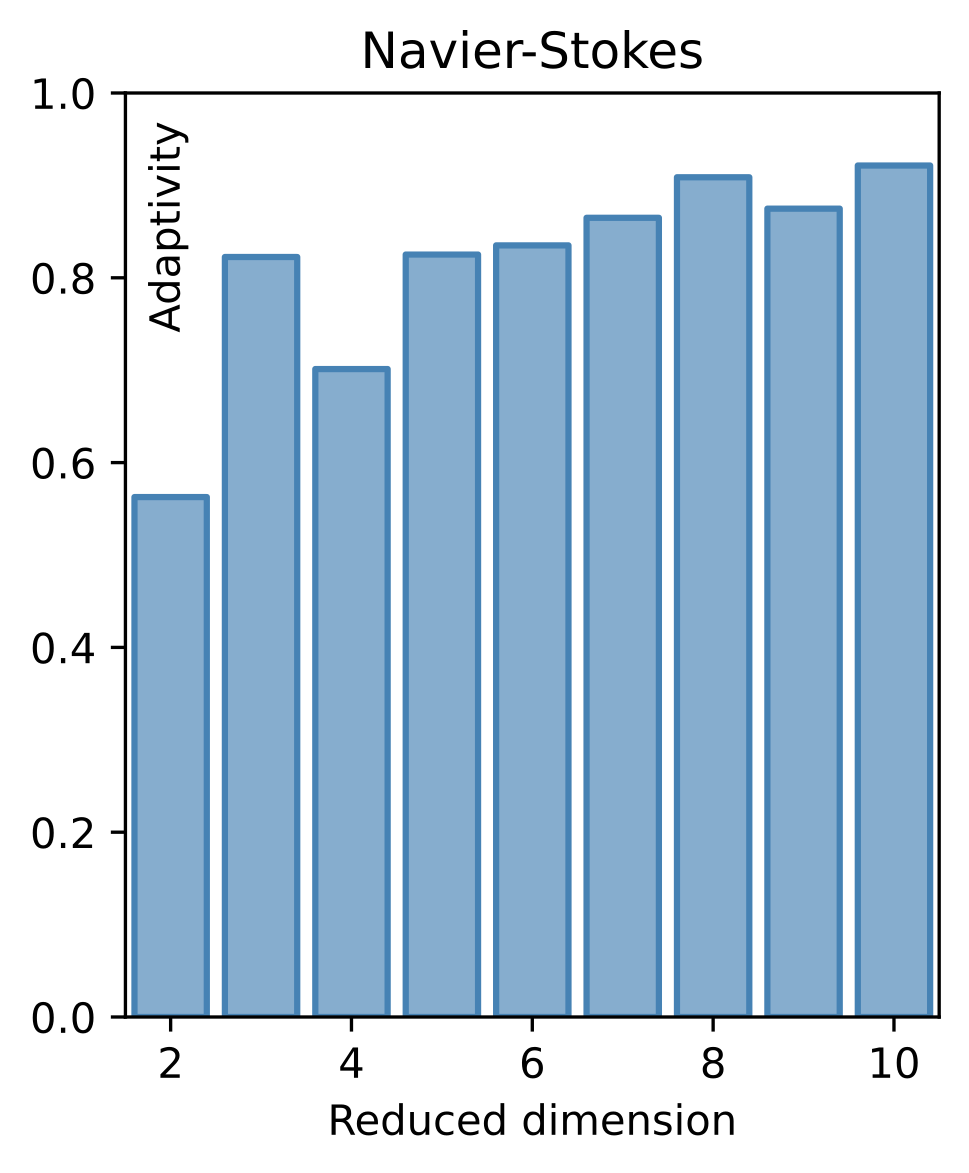}
    \caption{Adaptivity scores $\adapt(\dod)$ in the two case studies, for varying reduced dimension $n$. The scores are defined as in Eq. \eqref{eq:adapt-score}.}
    \label{fig:adaptivity}
\end{figure}

\begin{table}[ht!]
    \centering
    \begin{tabular}{lllll}\\\\
    \hline\hline
         \textbf{Case study} &  $\fomdim$ & $n$ & \textbf{Projection error} & \textbf{Adaptivity}\\\hline
         Eikonal Eq. & 9550 & 4 & 0.97\% & 0.56\\
         Navier-Stokes & 15202 & 4 & 4.32\% & 0.70\\
         %TBD & ??? & ?? & ?? & ??\\
         \hline\hline
    \end{tabular}
    \caption{DOD architectures selected for model order reduction (see Section \ref{sec:exp2}).}
    \label{tab:dod performances}
\end{table}

$$$$

$$$$

$$$$

$$$$

$$$$

\section{Deep Orthogonal Decomposition for reduced order modeling}
\label{sec:dodrom}
As we anticipated in Section \ref{sec:dod}, a DOD with $n$-modes allows us to reduce the complexity of the problem by shifting our attention from the parameter-to-solution map, $(\mub,\nub)\mapsto\ufomp\in\mathbb{R}^{\fomdim}$
to the parameter-to-DOD-coefficient map, i.e.
\begin{equation}\label{eq:reducedmap}(\mub,\nub)\mapsto \mathbf{c}_{\mub,\nub}:=\dod_{\mub}^{\top}\mass\ufomp\in\mathbb{R}^{n},\end{equation}
Since $n\ll N_{h}$, learning the latter should be much easier when compared to the original problem.
Before coming to our own proposal on \textit{how} learn \eqref{eq:reducedmap}, it is worth making a few considerations of general interest. We summarize them below.

\subsection{General considerations}
Let $\phi:\mathbb{R}^{\ngeo}\times\mathbb{R}^{\nphys}\to\mathbb{R}^{n}$ 
be any algorithm of choice, be it intrusive or data-driven, that, given $(\mub,\nub)$ seeks to approximate the corresponding DOD coefficient $\mathbf{c}_{\mub,\nub}.$ The latter naturally gives rise to a DOD-based ROM via the ansatz
$$\uromp:=\dod_{\mub}\cdot\phi(\mub,\nub)\approx\ufomp,$$
where "$\cdot$" emphasizes the presence of a matrix-vector multiplication. The quality of such an approximation will depend both on the DOD, $\dod$, and on the parameter-to-coefficient algorithm, $\phi$. To appreciate this, let
$$\mathcal{E}_{A}:=\mathbb{E}_{\mub,\nub}^{1/2}\|\ufomp-\uromp\|^{2},$$
be the approximation error of the whole ROM, here measured according to a root-mean-square-error metric (RMSE). The two architectures, $\dod$ and $\phi$, are responsible for the following sources of error
$$\mathcal{E}_{\textnormal{DOD}}:=\mathbb{E}_{\mub,\nub}^{1/2}\|\ufomp-\dod_{\mub}\dod_{\mub}^{\top}\mass\ufomp\|^{2},\quad\quad\mathcal{E}_{\textnormal{coeff}}:=\mathbb{E}_{\mub,\nub}^{1/2}|\mathbf{c}_{\mub,\nub}-\phi(\mub,\nub)^{2}|,$$
respectively. Here, $\mathcal{E}_{\text{DOD}}$ represents the DOD projection error, while $\mathcal{E}_{\text{coeff}}$ reflects the quality of the approximation of the reduced problem \eqref{eq:reducedmap}: together, these two quantities uniquely characterize the general expressivity of the ROM. In fact, it is straightforward to see that the following identity holds.
\begin{lemma} 
\label{lemma:error}
For all DOD networks and all reduced algorithms, one has
\begin{equation}
\label{eq:splitting}
\mathcal{E}_{A}^{2}=\mathcal{E}_{\textnormal{DOD}}^{2}+\mathcal{E}_{\textnormal{coeff}}^{2}.\end{equation}
\end{lemma}
\begin{proof}
    We shall prove the stronger identity below,
    \begin{equation}
\label{eq:splittingwise}
\|\ufomp-\uromp\|^{2}=\|\ufomp-\dod_{\mub}\dod_{\mub}^{\top}\mass\ufomp\|^{2}+|\mathbf{c}_{\mub,\nub}-\phi(\mub,\nub)|^{2}.\end{equation}
    holding for all $\mub\in\pgeo$ and all $\nub\in\pphys$. Note, in fact, that \eqref{eq:splitting} is just \eqref{eq:splittingwise} in expectation.
    To see that \eqref{eq:splittingwise} is valid, let $(\mub,\nub)\in\pgeo\times\pphys$. Since $\dod_{\mub}$ is orthonormal, we have
    $$|\mathbf{c}_{\mub,\nub}-\phi(\mub,\nub)|^{2}=\|\dod_{\mub}\mathbf{c}_{\mub,\nub}-\dod_{\mub}\phi(\mub,\nub)\|^{2}=\|\dod_{\mub}\dod_{\mub}^{\top}\mass\ufomp-\uromp\|^{2}.$$
    We now notice that, by definition, $$\left(\dod_{\mub}\dod_{\mub}^{\top}\mass\ufomp-\uromp\right)\in\text{span}\left(\dod_{\mub}\right).$$
    At the same time, by classical properties of linear projections,
    $$\left(\ufomp-\dod_{\mub}\dod_{\mub}^{\top}\mass\ufomp\right)\perp\text{span}\left(\dod_{\mub}\right).$$ Then, by orthogonality,
    \begin{multline*}
    \|\ufomp-\dod_{\mub}\dod_{\mub}^{\top}\mass\ufomp\|^{2}+\|\dod_{\mub}\dod_{\mub}^{\top}\mass\ufomp-\uromp\|^{2}=\\=\|\ufomp-\cancel{\dod_{\mub}\dod_{\mub}^{\top}\mass\ufomp}+\cancel{\dod_{\mub}\dod_{\mub}^{\top}\mass\ufomp}-\uromp\|^{2}=\\=\|\ufomp-\uromp\|^{2},\end{multline*}
    as claimed.
\end{proof}

This splitting shows that errors in the approximation of the reduced map propagate through the DOD in a stable way, that is: an error of $\epsilon$ in the approximation of the reduced coefficients is reflected in a corresponding error of (at most) $\epsilon$ at FOM level. We note that, typically, this property is exclusive to projection methods. Nonlinear techniques based on, e.g., autoencoders, instead, might suffer from error inflation. There, in fact, reduced coefficients are replaced by latent variables, and the lifting from $\mathbb{R}^{n}\to\mathbb{R}^{N_{h}}$ is obtained via a nonlinear decoder $\Psi.$ Consequently, errors at the latent level can be bounded, at most, as $$\|\Psi(\mathbf{c}_{\mub,\nub})-\Psi(\phi(\mub,\nub))\|\le L_{\Psi}|\mathbf{c}_{\mub,\nub}-\phi(\mub,\nub)|,$$ where $L_{\Psi}$ is the Lipschitz constant of the decoder module. In particular, if $L_{\Psi}>1$, errors may grow when passing through the decoder.

\subsection{Learning the DOD coefficients}
\label{subsec:dodnn}
The hybrid nature of the DOD projector opens up a wide spectrum of possibilities for computing DOD coefficients, ranging from intrusive to data-driven approaches. For example, during the \textit{online} phase, the DOD basis could be used to project and solve the governing equations, ultimately mimicking the idea underlying the POD-Galerkin ROMs. However, this approach would face major limitations when dealing, e.g., with nonlinear problems, as one would need to complement the DOD with a suitable hyperreduction strategy, or when facing operators with nonaffine dependency on the parameters, as that would quickly increase the online computational cost (in fact, in order to project the equations, one would still need to assemble the FOM first).

In light of this, and in order to be as general as possible, here we shall focus on non-intrusive strategies. %for the recovery of DOD coefficients. 
Given a candidate model class $\mathcal{C}\subset\{\tilde{\phi}:\mathbb{R}^{\ngeo}\times\mathbb{R}^{\nphys}\to\mathbb{R}^{n}\},$ which might consist of, e.g., neural network architectures, polynomials or Gaussian processes, the idea is to construct the reduced algorithm $\phi$ via mean-square regression, namely
$$\phi:=\argmin_{\tilde{\phi}\in\mathcal{C}}\frac{1}{\ntrain}\sum_{i=1}^{\ntrain}|\mathbf{c}_{\mub_{i}, \nub_{i}}-\tilde{\phi}(\mub_{i},\nub_{i})|^{2},$$
where $\mathbf{c}_{\mub_{i}, \nub_{i}}$ are defined according to \eqref{eq:reducedmap}.

In this work, we explore the use of neural network architectures, thus obtaining a ROM strategy that resambles the so-called POD-NN approach \cite{hesthaven2018non}, except for the presence of the adaptive DOD basis. In this sense, the following could be referred to as "DOD-NN". 
\\\\
The idea is to construct $\phi$ using a segregated architecture comprised of two submodules, $\phi_{1}$ and $\phi_{2}$, as to further differentiate between $\mub$ and $\nub$. More precisely, we design $\phi$ as
$$\phi(\mub,\nub):=\text{diag}\left[\phi_{1}(\mub)^{\top}\phi_{2}(\nub)\right],$$
where $\phi_{1}:\mathbb{R}^{\ngeo}\to\mathbb{R}^{m\times n}$ and $\phi_{2}:\mathbb{R}^{\nphys}\to\mathbb{R}^{m\times n}$ are two matrix-valued networks (implemented using classical architectures taking values in $\mathbb{R}^{mn}$, followed by a reshape layer). Mathematically speaking, this construction is equivalent to a separation of variables approach, where a function of two variables, $\mub$ and $\nub$, is expressed as the truncated sum (up to $m$ terms) of simpler functions. Similar strategies have also been explored elsewhere, as in, e.g., DeepONets \cite{lu2021learning} and POD-MINN \cite{vitullo2024nonlinear}. Here, the \textit{diag} operator is merely a matter of mathematical notation: in practice, we refrain from calculating the matrix product $\phi_{1}(\mub)^{\top}\phi_{2}(\nub)$ and instead compute the Hadamard product of $\phi_{1}(\mub)$ and $\phi_{2}(\nub)$, followed by a columnwise summation. 
\begin{figure}
    \centering
    \includegraphics[width=\textwidth]{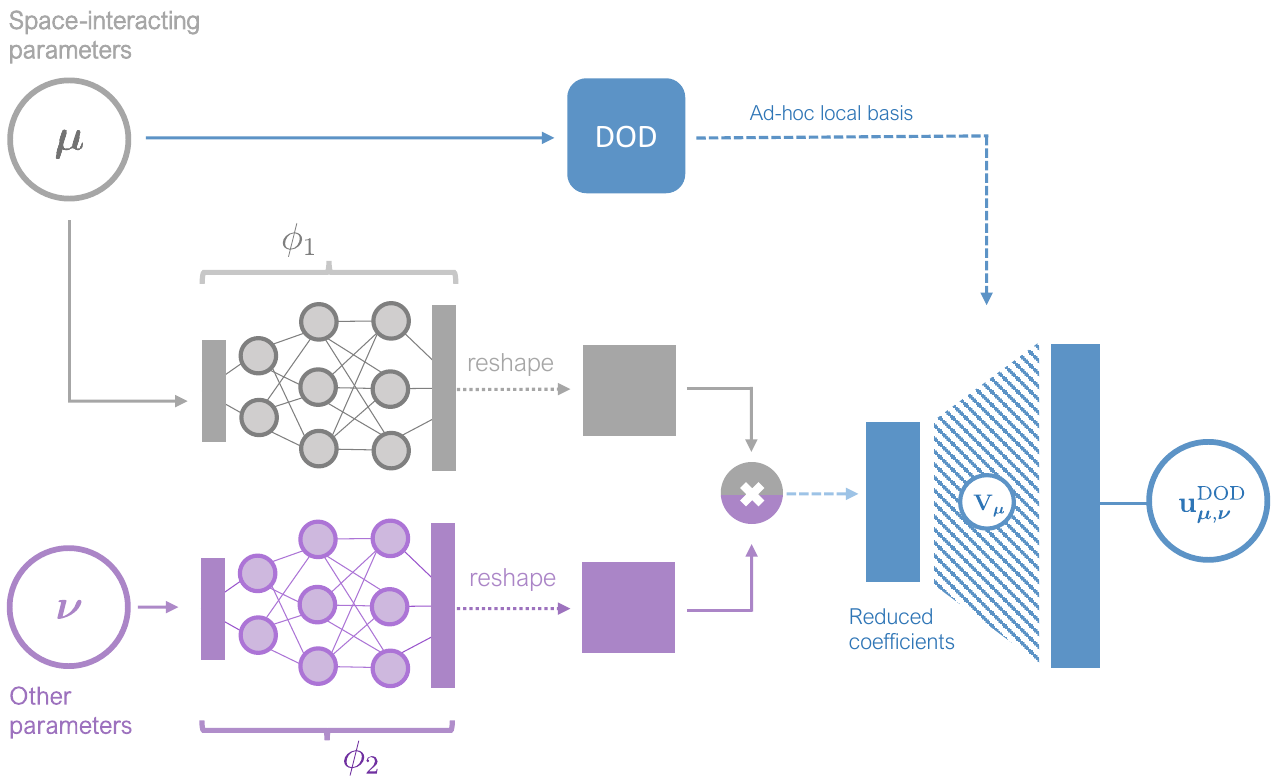}
    \caption{Sketch of the DOD-NN approach, Section \ref{subsec:dodnn}.}
    \label{fig:dod-nn}
\end{figure}
\;\\\\With this setup, the DOD-NN ROM, $\mathbf{u}_{\mub,\nub}^{\text{DOD-NN}}\approx\ufomp$, can be summarized in formulas as
\begin{multline}
\mathbf{u}_{\mub,\nub}^{\text{DOD-NN}}:=\dod_{\mub}\phi(\mub,\nub)=\\=\ambient\text{ORTH}\left(\left[R_{1}(s_{\mub}),\dots,R_{n}(s_{\mub})\right]\right)\cdot\text{diag}\left[\phi_{1}(\mub)^{\top}\phi_{2}(\nub)\right],
\end{multline}
or, visually, as in Figure \ref{fig:dod-nn}. In general, the accuracy of the approximation will depend on: the richness of the ambient space $\ambient$, the expressivity of the inner DOD module, $\tilde{\dod}$, and the quality of the parameter-to-DOD-coefficient approximation, $\phi$. In fact, it is straightforward to see that, as a direct consequence of Lemma \ref{lemma:equiv} and Lemma \ref{lemma:error}, the following error decomposition formula is given.

\begin{corollary}
    Let $(\mub,\nub)\mapsto\mathbf{u}_{\mub,\nub}^{\textnormal{DOD-NN}}$ be a \textnormal{DOD-NN} reduced order model with ambient matrix $\ambient$, inner DOD module $\tilde{\dod}$, and reduced network $\phi$. Then,
    \begin{align}  
    \label{eq:error-decomposition}
    \nonumber
    \mathbb{E}_{\mub,\nub}\|\ufomp-\mathbf{u}_{\mub,\nub}^{\textnormal{DOD-NN}}\|^{2}=\;\;&\mathbb{E}_{\mub,\nub}\|\ufomp-\ambient\ambient^{\top}\mass\ufomp\|^{2}+\\\nonumber&\mathbb{E}_{\mub,\nub}|\ambient^{\top}\mass\ufomp-\tilde{\dod}_{\mub}\tilde{\dod}^{\top}_{\mub}\ambient^{\top}\mass\ufomp|^{2}+\\&\mathbb{E}_{\mub,\nub}|\tilde{\dod}_{\mub}^{\top}\ambient^{\top}\mass\ufomp-\phi(\mub,\nub)|^{2},\end{align}
    where we recall that $\mathbf{u}_{\mub,\nub}^{\textnormal{DOD-NN}}:=\dod_{\mub}\phi(\mub,\nub)$ with $\dod_{\mub}:=\ambient\tilde{\dod}_{\mub}.$ The three terms at the right-hand-side of \eqref{eq:error-decomposition} are the (i) ambient error, the (ii) intrinsic DOD projection error and the (iii) coefficients error, respectively.
\end{corollary}

\begin{remark}
    Once a DOD architecture has been trained, replacing the orthonormalization block, ORTH, with a different one has no effect on the projection error. For instance, switching from a reduced QR algorithm to a Gram-Schidmt routine (and vice versa) has no impact on the accuracy of the DOD projection. In fact, as we noted in Section \ref{subsec:adaptivity}, the projection error depends only on the underlying subspace. However, from a practical perspective, this observation can be very useful. In fact, although changes in the ORTH block do not affect the DOD itself, they do have an impact on the DOD coefficients. For example, during our experiments, we observed the following: (a) when it comes to DOD training itself, a reduced QR module works best, as its remarkable efficiency can significantly speed up the optimization of the architecture; (b) when learning the DOD coefficients, instead, switching to a Gram-Schmidt block can be a more favorable option, as, empirically, it seems to facilitate training the reduced network $\phi$. 
\end{remark}

\section{Numerical experiments: model order reduction}
\label{sec:exp2}
We are now ready to extend the analysis presented in Section \ref{sec:exp1}, which was originally devoted to the sole purpose of dimensionality reduction, considering the application of the DOD-NN strategy for model order reduction. To do so, we shall consider the same case studies discussed in Section \ref{sec:exp1}, and, for each of them, proceed as follows.

First, we fix a reduced dimension $n$ and a corresponding DOD network. In doing so, we shall opt for a suitable compromise between: reconstruction accuracy, dimensionality reduction, and model volatility. Then, following the ideas presented in Section \ref{subsec:dodnn}, we implement and train a reduced network $\phi$. To do so, we rely on the same training data used for the DOD. Finally, we quantify the quality of the approximation by computing an empirical test error, calculated as
$$\textnormal{MRE}:=\frac{1}{N_{\textnormal{test}}}\sum_{i=1}^{N_{\textnormal{test}}}\frac{\|\ub_{\check{\mub}_{i},\check{\nub}_{i}}-\ub_{\check{\mub}_{i},\check{\nub}_{i}}^{\textnormal{DOD}}\|}{\|\ub_{\check{\mub_{i}},\check{\nub_{i}}}\|},$$
where we recall that $\ufomp^\text{DOD}:=\dod_{\mub}\phi(\mub,\nub)$, whereas $\{\check{\mub}_{i},\check{\nub}_{i}\}_{i=1}^{N_{\textnormal{test}}}$ are a collection of randomly sampled parameter configurations, drawn independently from the training set.

To better understand and appreciate the capabilities of the proposed approach, we also compare the performances of DOD-NN with two benchmark ROMs. In order to make the comparison as meaningful as possible, we first note the following. By construction, 
$$\ufomp^{\textnormal{DOD}}\in\spann(\ambient)$$
for all $\mub$ and all $\nub$: that is, the outputs of a DOD-NN module are always elements of the ambient space. In particular, the inner module of the DOD-NN architecture,
$$(\mub,\nub)\mapsto \tilde{\dod}_{\mub}\phi(\mub,\nub),$$
can be interpreted as approximating of the parameter-to-ambient-coefficients map. In light of this, it is perfectly reasonable to ask whether a classical POD-NN approach, with $\ambient$ as POD matrix, would provide better results, i.e. with the ansatz,
\begin{equation}
    \label{eq:pod-ambient}
    \ufomp\approx\ambient\phi_{\text{POD}}(\mub,\nub)
\end{equation}
where $\phi_{\text{POD}}:\mathbb{R}^{\ngeo}\times\mathbb{R}^{\nphys}\to\mathbb{R}^{N_{A}}$ is some neural network model. If so, the complex structure of the DOD-NN would be unmotivated, raising significant questions about the overall approach. Because of this, we believe this comparison to be highly valuable (also considering that the POD-NN approach is now a well-established technique in the ROM community, see, e.g., \cite{discacciati2024model, jacquier2021non, pichi2023artificial, wang2019non, xu2024predictions}). In view of these facts, we shall compare the performances of DOD-NN with the following benchmark models.
\begin{itemize}
    \item \textbf{Benchmark 1}. Starting from Eq. \eqref{eq:pod-ambient}, we design $\phi_{\text{POD}}$ as a classical DNN, taking as input the stacked vector of parameters $[\mub,\nub]\in\mathbb{R}^{\ngeo+\nphys}$.
    \item \textbf{Benchmark 2}. Here, we still rely on Eq. \eqref{eq:pod-ambient} but adopt a segregated architecture, $\phi(\mub,\nub)=\text{diag}\left[\phi_{1}(\mub)^{\top}\phi_{2}(\nub)\right]$, thus mimicking the idea in the DOD-NN approach.
\end{itemize}
In both cases, we train the benchmark models by minimizing the mean square error associated with the discrepancy $|\ambient^{\top}\mass\ufomp-\phi_{\text{POD}}(\mub,\nub)|.$ Furthermore, in order to make the comparison as fair as possible, we design the networks $\phi_{\text{POD}}$ to have the same complexity of the overall DOD-NN module. That is, we shall choose the number of layers and neurons so that $\phi_{\text{POD}}$ has approximately the same number of trainable parameters of the entire DOD-NN (thus, those of $\phi$ plus those of $\tilde{\dod}$).

\subsection{Results}
Table \ref{tab:dod performances} contains the main information on the DOD architectures selected for the model order reduction phase. In general, we have opted for those architectures showing a satisfactory accuracy but also moderate volatility. In fact, a higher variability of the DOD basis is typically reflected in a higher volatility of the DOD coefficients, $\mathbf{c}_{\mub}=\dod_{\mub}^{\top}\mass\ufomp$; thus, less volatile models are likely to yield representations that are simpler to learn.

\begin{table}
    \centering
    \begin{tabular}{lll}
    \hline\hline
         \textbf{ROM} &  \textbf{Eikonal Eq.} & \textbf{Navier-Stokes} \\\hline
         DOD-NN & 4.42\% & 5.26\% \\
         Benchmark 1 & 7.06\% &  6.83\% \\
         Benchmark 2 & 6.90\% & 9.52\% \\
         \hline\hline
    \end{tabular}
    \caption{ROM performances (relative $L^2$ error) for the three case studies presented in Section \ref{sec:exp1}. Benchmark models are POD-NN-like ROMs defined as in Eq. \eqref{eq:pod-ambient}.}
    \label{tab:rom performances}
\end{table}

A global overview of the final results is reported in Table \ref{tab:rom performances}. In particular, the DOD-NN approach consistently outperforms its POD-NN counterparts, at times exhibiting twice the accuracy. Overall, considering the inherently data-driven nature of the DOD-NN, we find that its performance is quite satisfactory, with relative errors always below 6\%. 

Interestingly, the superiority of DOD-NN is not only quantitative but also qualitative, as shown in Figures \ref{fig:elba-comparison}-\ref{fig:nstokes-comparison}. Here, we compare FOM solutions, DOD-NN approximations, and benchmark outputs for unseen values of the model parameters. When looking at these pictures, it is evident that the DOD-NN is much more aware of the geometrical parameters, effectively capturing their effect over the global solution field. On the contrary, its POD-NN counterparts are more likely to yield unphysical results: see, for example, Figure \ref{fig:nstokes-comparison}, where the benchmark model erroneously predicts the presence of fluid flow through the obstacle.

\begin{figure}
    \centering
    \hspace{0.75cm}Ground truth\hspace{3cm}DOD-NN
    \hspace{2.75cm}Benchmark 2\hfill\vspace{-1cm}
    \includegraphics[width=0.88\textwidth]{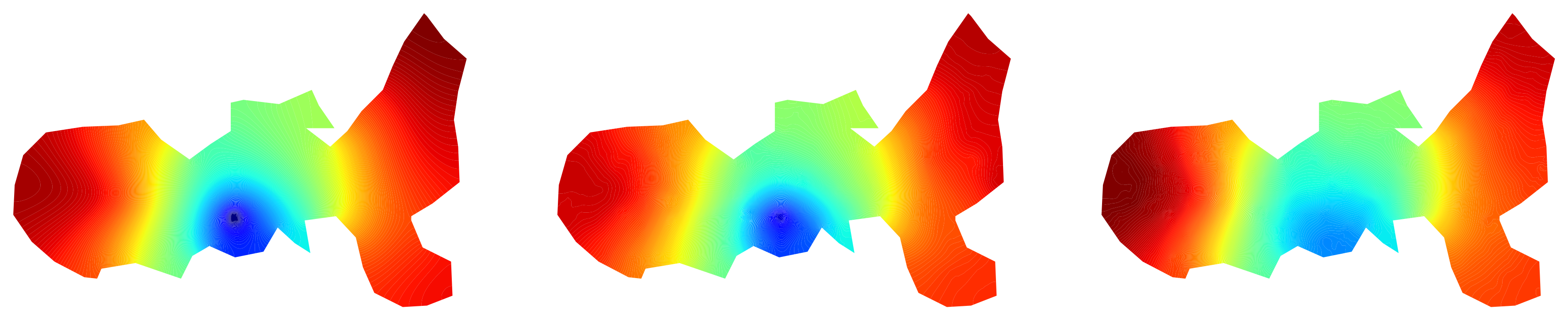}    \includegraphics[width=0.11\textwidth]{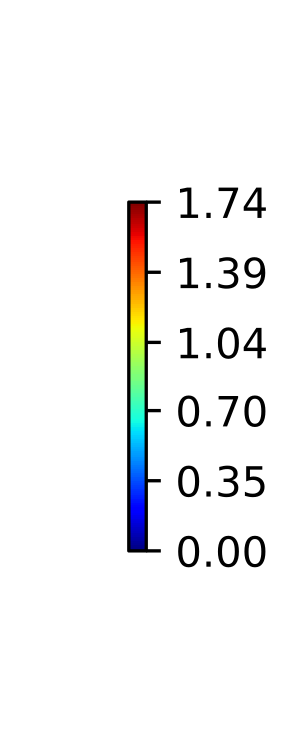}\vspace{-1cm}
    \includegraphics[width=0.88\textwidth]{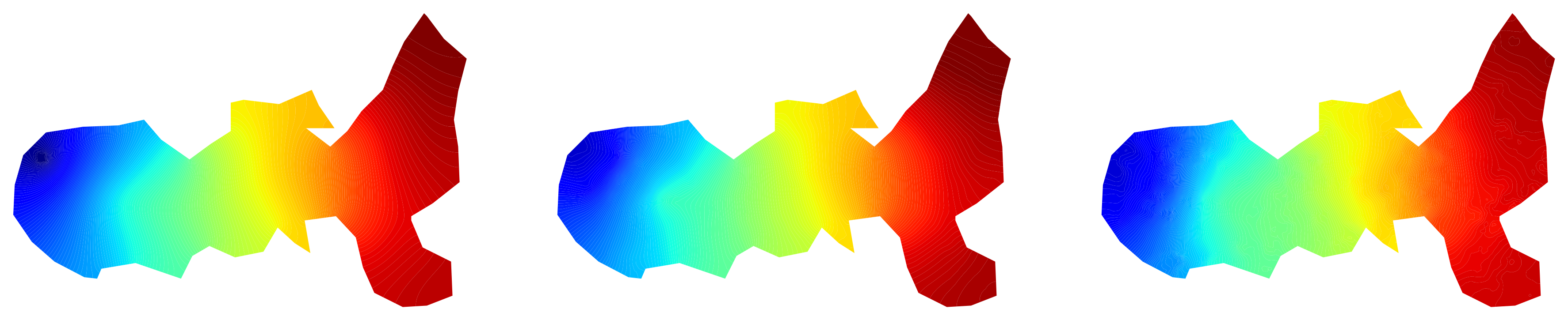}
    \includegraphics[width=0.11\textwidth]{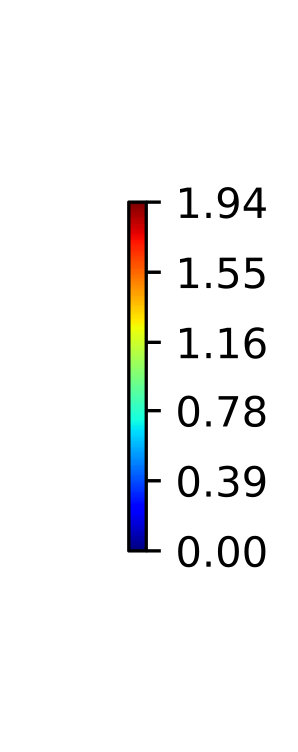}    
    \caption{Comparison between FOM (left), DOD-NN (center) and benchmark ROM (right) for two unseen configurations of the model parameters appearing in the Eikonal equation example. NB: only the best benchmark model is reported, cf. Table \ref{tab:rom performances}.}
    \label{fig:elba-comparison}
\end{figure}

\begin{figure}[tb]
    \centering
    Ground truth\hspace{2.7cm}DOD-NN
    \hspace{2.7cm}Benchmark 1
    \includegraphics[width=0.9\textwidth]{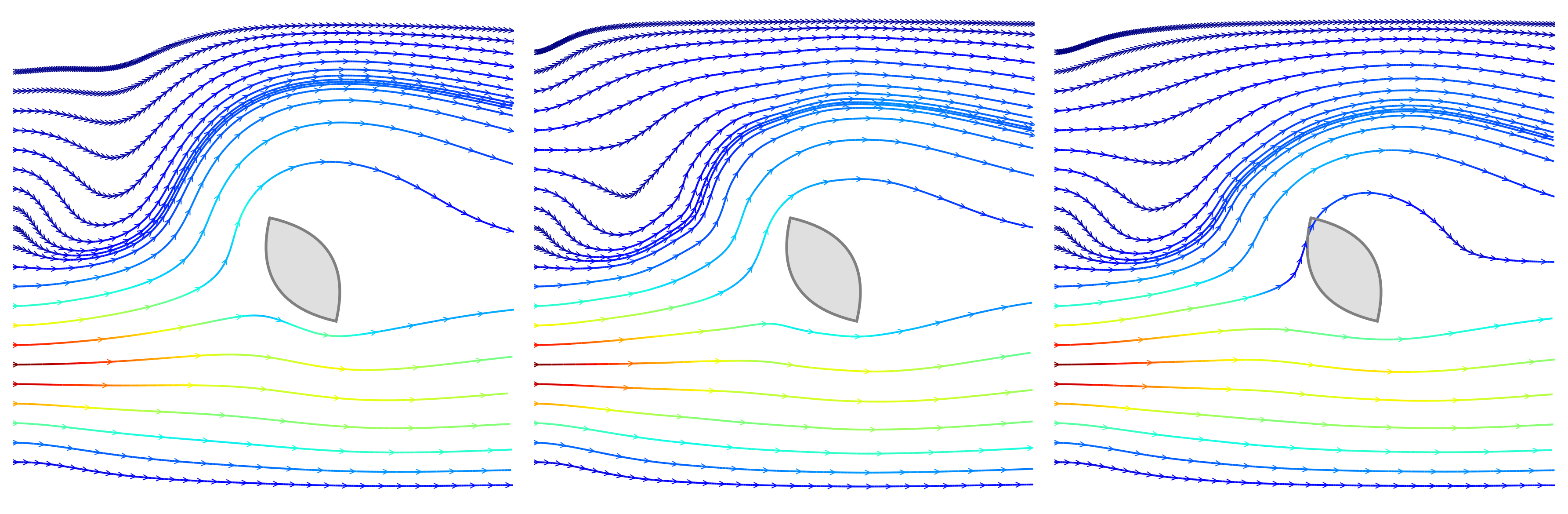}
    \includegraphics[width=0.9\textwidth]{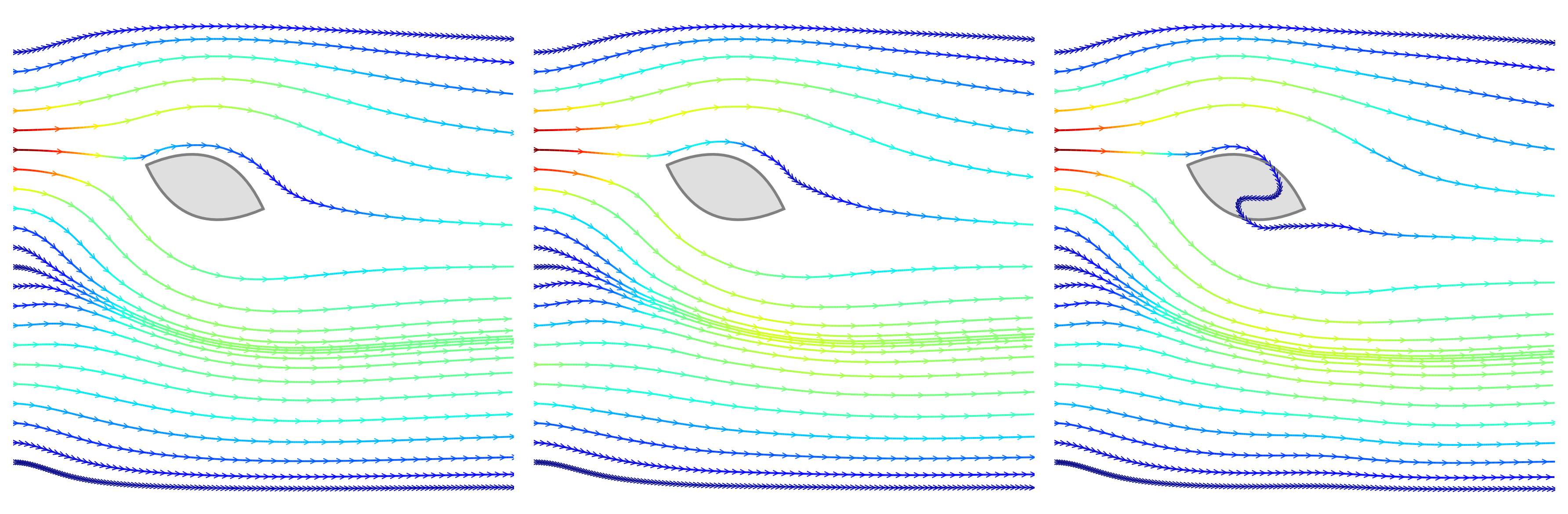}    
    \caption{Streamlines of the velocity fields for the Navier-Stokes example: comparison between FOM solutions (left), DOD-NN approximations (center) and benchmark outputs (right) for two unseen configurations of the model parameters $(\mub,\nub)$. Different colors indicate different magnitudes of the velocity field (lower values in blue, higher values in red). NB: only the best benchmark model is reported, cf. Table \ref{tab:rom performances}.}
    \label{fig:nstokes-comparison}
\end{figure}

\section{Conclusions}
\label{sec:conclusions}
We presented a novel approach, termed DOD, for dimensionality reduction and reduced-order modeling of parametrized stationary PDEs, which, by leveraging on deep neural networks, can effectively approximate the solution manifold through a continuously adaptive local basis. Compared to other existing approaches, the DOD stands out for its ability to handle problems with a slow decaying Kolmogorov $n$-width, providing both rich and interpretable latent representations. Additionally, DOD projections can be combined with both intrusive and non-intrusive techniques, allowing domain practitioners to freely choose between physics-based and data-driven ROMs. For instance, in this work, we explored the use of a fully non-intrusive ROM, termed DOD-NN.

In this work, we focused exclusively on stationary PDEs. In principle, integrating time as an additional parameter (to be coupled with $\mub$) should make the extension of our approach to time-dependent problems straightforward. However, a more in-depth exploration may unveil more advanced strategies, potentially leading to more sophisticated DOD-based ROMs capable of enforcing specific properties inherent to dynamical systems (Markovianity, iterative structure, etc.).
Parallel to this, another interesting question would be to investigate suitable generalizations of the DOD capable of tackling PDEs with parameter-dependent spatial domains. Right now, this is only possible in tandem with fictitious domain or mesh deformation approaches. In this concern, radically different tools, such as graph neural networks, might offer valuable insights. Similarly, another open question is whether one can \textit{discover} the decoupling of the parameter space from data, i.e. how to distinguish between $\mub$ and $\nub$, without relying on prior knowledge, similarly to what happens with slow-fast decompositions \cite{zielinski2022discovery}.
We leave these considerations for future work.

\section*{Acknowledgments}
Paolo Zunino acknowledges the support of the grant MUR PRIN 2022 No. 2022WKWZA8 "Immersed methods for multiscale and multiphysics problems (IMMEDIATE)” part of the Next Generation EU program. Andrea Manzoni acknowledges the PRIN 2022 Project “Numerical approximation of uncertainty quantification problems for PDEs by multi-fidelity methods (UQ-FLY)” (No. 202222PACR), funded by the European Union - NextGenerationEU and the project FAIR (Future Artificial Intelligence Research), funded by the NextGenerationEU program within the PNRR-PE-AI scheme (M4C2, Investment 1.3, Line on Artificial Intelligence). 
Nicola R. Franco and Paolo Zunino acknowledge the project Cal.Hub.Ria (Piano Operativo Salute, traiettoria 4), funded by MSAL.
The present research is part of the activities of project Dipartimento di Eccellenza 2023-2027, Department of Mathematics, Politecnico di Milano, funded by MUR. NRF, AM and PZ are members of the Gruppo Nazionale per il Calcolo Scientifico (GNCS) of the Istituto Nazionale di Alta Matematica (INdAM). 

%\bibliographystyle{elsarticle-harv} 

\normalem
\printbibliography


%\begin{thebibliography}{00}



%\end{thebibliography}

\appendix

\section{Proof of Theorem \ref{theorem:dod}}
\label{sec:appendix:proof}
In what follows, with little abuse of notation, we write $|\cdot|$ to indicate both the Euclidean norm and the Frobenius norm.\\\\
\noindent\textit{Proof}. Fix $n\in\mathbb{N}$. Let $\lambda>0$ be the smallest eigenvalue of the positive definite matrix $\mass$. To start, we note that, if $\VV$ is a matrix whose columns are orthonormal with respect to $\mass$, then all the entries of $\VV$ must be smaller, in modulus, than $1/\sqrt{\lambda}.$ To see this, let $\mathbb{I}_{n}$ be the $n\times n$ identity matrix and assume that
$$\VV^{\top}\mass\VV=\mathbb{I}_{n}.$$
Let $\mathbb{L}:=\mass^{1/2}$ be the unique positive definite matrix for which $\mathbb{L}^{\top}\mathbb{L}=\mass,$ so that its smallest eigenvalue equals $\sqrt{\lambda}.$ Fix any entry $v_{i,j}$ in $\VV$ and let $\mathbf{v}_{i}$ be the corresponding column. We have
$$1=\mathbf{v}_{i}^{\top}\mass\mathbf{v}_{i}=|\mathbf{L}\mathbf{v}_{i}|^{2}\ge\lambda|\mathbf{v}_{i}|^{2}\ge\lambda v_{i,j}^{2},$$
implying $v_{i,j}\le 1/\sqrt{\lambda}.$ Having this in %our minds, 
mind, we shall assume, without loss of generality, that $\lambda=1$ (if not, we may always rescale).
\\\\
Let $J:\pgeo\times[-1,1]^{\fomdim\times n}\to\mathbb{R}$ be the following objective functional
$$J(\mub,\mathbb{V})=\mathbb{E}_{\nub}\|\ufomp-\mathbb{V}\mathbb{V}^{\top}\mass\ufomp\|,$$
and fix any $\varepsilon>0$. We recall that, since $J$ is continuous and $\pgeo\times[-1,1]^{\fomdim}$ is compact, $J$ is uniformly continuous. In particular, there exists $\delta>0$ such that
$$|\mub-\mub'|+|\mathbb{V}-\mathbb{V}'|<\delta\implies|J(\mub,\mathbb{V})-J(\mub',\mathbb{V}')|<\varepsilon.$$
Then, classical results on measurable selections (see, e.g., Lemma \ref{lemma:optselec}) show that there exists a Borel measurable map $s:\pgeo\to[-1,1]^{\fomdim\times n}$ acting as
$$s:\mub\to\argmin_{\mathbb{V}\in[-1,1]^{N_{h}\times n}}J(\mub,\mathbb{V}).$$
Since $s$ is bounded, it follows that $s\in L^{1}(\pgeo;\mathbb{R}^{\fomdim\times n})$ in the Bochner sense. In particular, there exists a deep ReLU neural network $\mathbf{V}:\pgeo\to\mathbb{R}^{\fomdim\times n}$ such that $\mathbb{E}|\mathbf{V}_{\mub}-s(\mub)|<\delta$; furthermore, we can assume, without loss of generality, that $\mathbf{V}$ takes values in $[-1,1]^{\fomdim\times n}$, see, e.g., Lemma \ref{lemma:density}. 

By uniform continuity of $J$ it follows that $$|J(\mub,s(\mub))-J(\mub,\mathbf{V}_{\mub})|<\varepsilon\quad\quad\forall\mub\in\pgeo.$$
Consequently, as $\mathbb{E}_{\mub,\nub}\|\ufomp-\mathbf{V}_{\mub}\mathbf{V}_{\mub}^{\top}\mass\ufomp\|=\mathbb{E}_{\mub}\left[J(\mub,\mathbf{V}_{\mub})\right]$, we have
$$\mathbb{E}_{\mub,\nub}\|\ufomp-\mathbf{V}_{\mub}\mathbf{V}_{\mub}^{\top}\mass\ufomp\|\le\mathbb{E}_{\mub}\|J(\mub,\mathbf{V}_{\mub})-J(\mub,s(\mub))\| + \mathbb{E}_{\mub}\left[J(\mub,s(\mub))\right]$$
 \begin{equation*}
\implies \mathbb{E}_{\mub,\nub}\|\ufomp-\mathbf{V}_{\mub}\mathbf{V}_{\mub}^{\top}\mass\ufomp\|<\varepsilon+\mathbb{E}_{\mub}\left[J(\mub,s(\mub))\right].\hspace{9.5em}\;
 \end{equation*}
Since
\begin{multline*}
\mathbb{E}_{\mub}\left[J(\mub,s(\mub))\right]=\mathbb{E}_{\mub}\left[\min_{\mathbb{V}}J(\mub,\mathbb{V})\right]\le\\\le\mathbb{E}_{\mub}\left[\min_{\mathbb{V}}\sup_{\nub}\|\ufomp-\mathbb{V}\mathbb{V}^{\top}\mass\ufomp\|\right]=\mathbb{E}_{\mub}\left[d_{n}(\solmanifold_{\mub})\right],
\end{multline*}
the conclusion follows.\hfill$\square$

\newtheoremstyle{lemmaA}% Numbered
{18pt plus2pt minus1pt}% Space above
{18pt plus2pt minus1pt}% Space below
{\itshape}% Body font
{0pt}% Indent amount
{\bfseries}% Theorem head font
{.}% Punctuation after theorem head
{.5em}% Space after theorem headi
{\thmname{#1} A\thmnumber{#2}%
  \thmnote{ {\the\thm@notefont(#3)}}}% Theorem head spec (can be left empty, meaning `normal')
\theoremstyle{lemmaA}
\newtheorem{lemmaA}{Lemma}
\labelformat{lemmaA}{A#1}

\begin{lemmaA}
    \label{lemma:density}
    Let $\mathbb{P}$ be a probability distribution over $\mathbb{R}^{p}.$ Let $\boldsymbol{s}:\mathbb{R}^{p}\to[-1,1]^{q}$ be a measurable map. Then, for every $\delta>0$, there exists a ReLU deep neural network $\boldsymbol{v}:\mathbb{R}^{p}\to\mathbb{R}^{q}$ such that
    $$\mathbb{E}|\boldsymbol{v}- \boldsymbol{s}|<\delta,$$
    where $\mathbb{E}$ denotes the expectation operator with respect to $\mathbb{P}.$ Furthermore, $\boldsymbol{v}$ can be chosen so that $\boldsymbol{v}\left(\mathbb{R}^{p}\right)\subseteq[-1,1]^{q}.$ 
\end{lemmaA}
\begin{proof}
    By Hornik's Theorem \cite{hornik1991approximation}, there exists a ReLU deep neural network $\boldsymbol{v}_{0}:\mathbb{R}^{p}\to\mathbb{R}^{q}$ such that $\mathbb{E}|\boldsymbol{v}_{0}-\boldsymbol{s}|<\delta$. Let now $\rho$ denote the ReLU activation function and consider the three-layer network $L:\mathbb{R}^{q}\to\mathbb{R}^{q}$ given by $$L(\xb)=\mathbf{e}-\mathbb{I}_{q}\rho(-\mathbb{I}_{q}\rho(\mathbb{I}_{q}\xb+\mathbf{e})+2\mathbf{e}),$$
    where $\mathbb{I}_{q}$ is the $q\times q$ identity matrix, and $\mathbf{e}:=[1,\dots,1]^{\top}\in\mathbb{R}^{q}$, so that, with little abuse of notation, one has $L(\xb)=1-\rho(2-\rho(\xb+1)).$
    The latter acts as follows: given any input $\mathbf{x}\in\mathbb{R}^{q}$, it leaves unchanged all those entries that they lie in $[-1,1]$, while it squashes the rest to $\pm1$. Then, it is straightforward to see that $\boldsymbol{v}:=L\circ\boldsymbol{v}_{0}$ fulfills all the desired properties. In fact, by construction, $\boldsymbol{v}(\mub)\in[-1,1]^{q}$ for all $\mub\in\mathbb{R}^{p}$; furthermore,
    $$\mathbb{E}|\boldsymbol{v}-\boldsymbol{s}|=\mathbb{E}|L\circ\boldsymbol{v}_{0}-L\circ\boldsymbol{s}|\le \mathbb{E}|\boldsymbol{v}_{0}-\boldsymbol{s}|<\delta,$$
as $L$ is 1-Lipschitz.
\end{proof}

\section{Measurable selections}
\label{sec:appendix:measurable}
\newtheoremstyle{lemmaB}% Numbered
{18pt plus2pt minus1pt}% Space above
{18pt plus2pt minus1pt}% Space below
{\itshape}% Body font
{0pt}% Indent amount
{\bfseries}% Theorem head font
{.}% Punctuation after theorem head
{.5em}% Space after theorem headi
{\thmname{#1} B\thmnumber{#2}%
  \thmnote{ {\the\thm@notefont(#3)}}}% Theorem head spec (can be left empty, meaning `normal')
\theoremstyle{lemmaB}
\newtheorem{definitionB}{Definition}
\labelformat{definitionB}{B#1}
\newtheorem{lemmaB}{Lemma}
\labelformat{lemmaB}{B#1}

We recall that a \textit{Polish space} is a complete separable metric space: in particular, all separable Banach spaces are Polish spaces. Hereon, we shall also use the notation $2^{X}$ to denote the power set of a given set $X$, that is, the collection of all its possible subsets, $2^{X}:=\{A\;\;|\;\;A\subseteq X\}.$

\begin{definitionB}
\label{def:setvalued}
Let  $(Y,\mathscr{M})$ be a measurable space and let $(X,d)$ be a Polish space. Let $F:Y\to 2^{X}$. We say that $F$ is a measurable set-valued map if the following conditions hold:
\begin{itemize}
    \item [a)] $F(y)$ is closed in $X$ for all $y\in Y$;\\
    \item [b)] for all open sets $A\subseteq X$ one has $\mathcal{S}_{A}\in\mathscr{M}$, where
\begin{equation}
    \label{eq:sets}
    \mathcal{S}_{A}:=\{y\in Y\;|\;F(y)\cap A\neq\emptyset\}.
\end{equation}
\end{itemize}
\end{definitionB}

\begin{lemmaB}
    \label{lemma:optselec}
    Let $(X,d_{X})$ be a Polish space and let $(C,d_{C})$ be a compact metric space. Let $J:X\times C\to \mathbb{R}$ be continuous. Then, there exists a Borel measurable map $f:X\to C$ such that
    $$J(x,f(x))=\min_{c\in C}J(x,c)\quad\quad\forall x\in X.$$  
\end{lemmaB}

\begin{proof} 
To start, we note that the statement in the Lemma is well-defined as for all $x\in X$ one has
$$\inf_{c\in C}J(x,c)=\min_{c\in C}J(x,c),$$
due compactness of $C$ and continuity of $c\mapsto J(x,c)$. Let now $F:X\to 2^{C}$ be the following set-valued map
$$F:x\mapsto\left\{c\in C\;\text{such that}\;J(x,c)=\min_{c'\in C}J(x,c')\right\},$$
so that $F$ assigns a nonempty subset of $C$ to each $x\in X$. We aim at showing that $F$ is a measurable set-valued map as in Definition \ref{def:setvalued}. Thus, we start by noting that $F(x)\subseteq C$ is closed in $C$ for all $x\in X$. In fact, for any fixed $x\in X$ we have 
$$F(x)=j_{x}^{-1}\left(\left\{\min_{c'\in C}j_{x}(c')\right\}\right),$$
where $j_{x}:=J(x,\cdot)$. Thus, $F(x)$ is closed as it is the pre-image of a singleton under the continuous transformation $j_{x}$. Then, following Definition \ref{def:setvalued}, we are left to show that for any open set $A\subseteq C$, the set
$$\mathcal{S}_{A}:=\{x\in X\;:\;F(x)\cap A\neq\emptyset\}$$
is Borel measurable. To this end, we shall first prove that \begin{equation}
    \label{eq:compact-closed}
    K\subseteq C\;\text{compact}\implies S_{K} \;\text{closed}.
\end{equation}
Then, our conclusion would immediately follow, as any open set $A\subseteq C$ can be written as the countable union of compact sets, $A=\cup_{n\in\mathbb{N}}K_{n}$, and clearly $\mathcal{S}_{A}=\cup_{n}\mathcal{S}_{K_{n}}.$ To see that \eqref{eq:compact-closed} holds, let us fix any compact set $K\subseteq C$. Let $\{x_{n}\}_{n}\subseteq \mathcal{S}_{K}$ be a sequence converging to some $x\in X$. By definition of $\mathcal{S}_{K}$, for each $x_{n}$ there exists a $c_{n}\in K$ such that $c_{n}\in F(x_{n})$, i.e. for which $J(x_{n},c_{n})=\min_{c'} J(x_{n},c')$. Since $K$ is compact, up to passing to a subsequence, there exists some $c\in K$ such that $c_{n}\to c$. Let now $\tilde{c}\in C$ be a minimizer for $x$, i.e. a suitable element for which $J(x,\tilde{c})=\min_{c'\in C}J(x,c')$. By continuity, we have
\begin{equation*}    J(x,c)=\lim_{n\to+\infty}J(x_{n},c_{n})=\lim_{n\to+\infty}\min_{c' \in C}J(x_{n},c')\le\lim_{n\to+\infty}J(x_{n},\tilde{c})=J(x,\tilde{c}),
\end{equation*}
implying that $c$ is also a minimizer for $x$. It follows that $c\in K\cap F(x)$ and thus $x\in \mathcal{S}_{K}$. In particular, $\mathcal{S}_{K}$ is closed. 

All of this shows that $F$ fulfills the requirements of Definition \ref{def:setvalued}, making it a measurable set-valued map. We are then allowed to invoke the celebrated measurable selection theorem by Kuratowski–Ryll-Nardzewski \cite{aubin2009set}, which ensures the existence of a measurable map $f:X\to C$ such that $f(x)\in F(x)$, i.e. $J(x,f(x))=\min_{c\in C}J(x,c)$, as wished.    
\end{proof}

\section{Additional details on neural network architectures}
\label{sec:appendix:architectures}

We report below the architectures of the autoencoder models in Section~\ref{sec:exp1}, together with the neural network models employed in Section~\ref{sec:exp2} (DOD-NN and POD-NN benchmarks). NB: all architectures employ the 0.1-leakyReLU activation at the \textit{internal} layers.

\subsection*{Autoencoder architectures}
\vspace{0.2cm}
\noindent \small{Table C.1: Architecture for the Eikonal equation, Section \ref{subsec:eikonal}.}\vspace{0.2cm}\\
\noindent\begin{tabular}{lll}
\hline\hline
    \textbf{Component} &  \textbf{Specifics} & \textbf{Terminal activation}\\\hline
    Encoder & $N_{A} \mapsto n $ & 0.1-leakyReLU\\
    Decoder & $n \mapsto 100 \mapsto 100 \mapsto N_{A}$& -\\
    \hline\hline\\
\end{tabular}
\;\\\\
\small{Table C.2: Architecture for the Navier-Stokes equation, Section \ref{subsec:navier-stokes}.}\vspace{0.2cm}\\
\begin{tabular}{lll}
\hline\hline
    \textbf{Component} &  \textbf{Specifics} & \textbf{Terminal activation}\\\hline
    Encoder & $N_{A} \mapsto n $ & 0.1-leakyReLU\\
    Decoder & $n \mapsto 1000 \mapsto N_{A}$& -\\
    \hline\hline
\end{tabular}

\subsection*{Neural network architectures for the Eikonal equation (Section \ref{sec:exp2})}
\vspace{0.2cm}
\noindent \small{Table C.3: POD-NN network (Benchmark 1).}\vspace{0.2cm}\\
\begin{tabular}{ll}
\hline\hline
    \textbf{Specifics} & \textbf{Terminal activation}\\\hline
    $(p+p') \mapsto 200 \mapsto 200 \mapsto N_{A} $ & -\\
    \hline\hline\\
\end{tabular}
\;\\\\
\small{Table C.4: POD-NN network (Benchmark 2).}\vspace{0.2cm}\\
\begin{tabular}{lll}
\hline\hline
    \textbf{Component} &  \textbf{Specifics} & \textbf{Terminal activation}\\\hline
    $\phi_{1}$ & $p \mapsto 10 \mapsto 25\times N_{A} $ & 0.1-leakyReLU\\
    $\phi_{2}$ & $p' \mapsto 10 \mapsto 25\times N_{A}$& -\\
    \hline\hline\\
\end{tabular}
\;\\\\
\small{Table C.5: DOD-NN networks.}\vspace{0.2cm}\\
\begin{tabular}{lll}
\hline\hline
    \textbf{Component} &  \textbf{Specifics} & \textbf{Terminal activation}\\\hline
    $\phi_{1}$ & $p \mapsto 40 \mapsto 5\times n $ & 0.1-leakyReLU\\
    $\phi_{2}$ & $p' \mapsto 40 \mapsto 5\times n$& -\\
    \hline\hline\\
\end{tabular}

\subsection*{Neural network architectures for the Navier-Stokes 
equation (Section \ref{sec:exp2})}
\noindent\normalsize{In what follows, $\stackrel{*}{\mapsto}$ denotes a nonlearnable feature layer acting either as $$[\theta, x_{0},y_{0},\alpha,\beta]\stackrel{*}{\mapsto}[\cos4\theta,\sin4\theta,x_{0},y_{0},\alpha,\beta]$$ or $[\theta, x_{0},y_{0}]\stackrel{*}{\mapsto}[\cos4\theta,\sin4\theta,x_{0},y_{0}]$, depending on the input size. As in Section \ref{sec:exp1}, the latter is used to enforce rotational symmetry.}
\\\\
\small{Table C.6: POD-NN network (Benchmark 1).}\vspace{0.2cm}\\
\begin{tabular}{ll}
\hline\hline
    \textbf{Specifics} & \textbf{Terminal activation}\\\hline
    $(p+p')\textcolor{white}{p}\stackrel{*}{\mapsto}\textcolor{white}{6}(4+p') \mapsto 170 \mapsto 170 \mapsto N_{A} $ & -\\
    \hline\hline\\
\end{tabular}
\;\\\\
\small{Table C.7: POD-NN network (Benchmark 2).}\vspace{0.2cm}\\
\begin{tabular}{lll}
\hline\hline
    \textbf{Component} &  \textbf{Specifics} & \textbf{Terminal activation}\\\hline
    $\phi_{1}$ & $p\;\stackrel{*}{\mapsto} 4 \mapsto 30 \mapsto 5\times N_{A} $ & 0.1-leakyReLU\\
    $\phi_{2}$ & $p' \mapsto 30 \mapsto 5\times N_{A}$& -\\
    \hline\hline\\
\end{tabular}
\;\\\\
\small{Table C.8: DOD-NN networks.}\vspace{0.2cm}\\
\begin{tabular}{lll}
\hline\hline
    \textbf{Component} &  \textbf{Specifics} & \textbf{Terminal activation}\\\hline
    $\phi_{1}$ & $p\;\stackrel{*}{\mapsto} 4 \mapsto 50 \mapsto 25\times n $ & 0.1-leakyReLU\\
    $\phi_{2}$ & $p' \mapsto 50 \mapsto 25\times n$& -\\
    \hline\hline
\end{tabular}

\end{document}